\documentclass{article}
\usepackage[utf8]{inputenc}
\usepackage{amsmath}
\usepackage{mathtools}
\usepackage{amsfonts}
\usepackage{amssymb}
\usepackage{amsthm}
\usepackage[numbers,sort&compress]{natbib}

\usepackage{cases}

\DeclareMathOperator{\supp}{Supp}

\newcommand{\dt}{\partial_t}
\newcommand{\dx}{\partial_x}
\newcommand{\dz}{\partial_z}

\newcommand{\idx}{\,d\vec{x}}
\newcommand{\dv}{\mathrm{div}}
\newcommand{\dvh}{\mathrm{div}_h}

\newcommand{\subeqref}[2]{$ \eqref{#1}_{#2} $}
\newcommand{\norm}[2]{\Arrowvert #1 \Arrowvert_{#2}}
\newcommand{\Lnorm}[1]{L^{#1}}
\newcommand{\Hnorm}[1]{H^{#1}}

\newtheorem{lm}{Lemma}

\newtheorem{thm}{Theorem}
\newtheorem{remark}{Remark}
\newtheorem{definition}{Definition}

\allowdisplaybreaks

\title{Well-posedness of strong solutions to the anelastic equations of stratified viscous flows
 }
\author{Xin Liu\footnote{Department of Mathematics, Texas A{\&}M University, College Station, TX 77843, USA. Email: stleonliu@gmail.com} \,\, and \,  Edriss S. Titi\footnote{Department of Mathematics, Texas A{\&}M University, College Station,  TX 77840, USA.  Department of Applied Mathematics and Theoretical Physics, University of Cambridge, Cambridge CB3 0WA, UK.
		Department of Computer Science and Applied Mathematics, Weizmann Institute of Science, Rehovot 76100, Israel. Email: {titi@math.tamu.edu}\, and \, {Edriss.Titi@damtp.cam.ac.uk}}
}

\date{June 26, 2019}

\begin{document}

\maketitle

\begin{abstract}
	We establish the local and global well-posedness of strong solutions to the two-- and three-dimensional anelastic equations of stratified viscous flows.
	In this model,
	the interaction of the density profile with the velocity field is taken into account, and the density background profile is permitted to have physical vacuum singularity.
	The existing time of the solutions is infinite in two dimensions, with general initial data, and in three dimensions with small initial data.
\end{abstract}
{\par\noindent {\bf Keywords:} anelastic approximation; well-posedness; physical vacuum; stratified flows}

\section{Introduction}

The anelastic Navier-Stokes system for stratified flows,
\begin{equation}\label{eq:2d-anel-NS}
    \begin{cases}
        \rho (\dt u + u \cdot \nabla u) + \rho \nabla p = \Delta u & \text{in} ~ \Omega, \\
        \dv (\rho u) = 0 & \text{in} ~ \Omega,
    \end{cases}
\end{equation}
is derived as the limiting system of the compressible Navier-Stokes system after filtering out the acoustic waves for strong stratified flows.
Here the velocity field $ u $ and the pressure $ p $ are the unknowns while the background density $ \rho $ is given as a time-independent, negative function.
The rigorous derivation of \eqref{eq:2d-anel-NS} can be found in \cite{Masmoudi2007}. Comparing to the incompressible Navier-Stokes system (see, e.g., \cite{Temam1984,Constantin1988}), the main difference is the incompressible condition $ \dv u = 0 $ is replaced by the anelastic relation $ \dv ( \rho u ) = 0 $ with the background density profile $ \rho $, which represents the strong stratification owing to the balance of the gravity and the pressure (see, e.g., \cite{FeireislSingularLimits}). Such an approximation preserves slight compressibility while filtering out the acoustic waves, which significantly simplifies the original compressible Navier-Stokes system, and enables more efficient computation applications to relevant model flows in physical reality. In particular, the anelastic approximation is used to describe the semi-compressible ocean dynamics (see, e.g., \cite{Dewar2015,Dewar2016}), as well as the tornado-hurricane dynamics (see, e.g., \cite{Nolan2002,Saal2010}). We refer the readers to \cite{Ogura1962,Klein2004,Majda2003,Braginsky,Bois2006,Feireisl2008,klein2010regime,klein2016doubly} for related topics and comparisons of various models of the atmospheric and oceanic dynamics.

We remark that the background density profile $ \rho $ in the anelastic relation $ \dv(\rho u) = 0 $ is given by the resting state $ \nabla P(\rho) = \rho g \vec{e}_z $, where $ P(\rho) $ denotes the pressure potential and $ g $ is the gravity acceleration. For the sake of simplifying the presentation, we have choosed the gravity to point upwards, which can be done after performing a vertical reflection of the coordinates.
In the case when the flow connects to vacuum continuously, the resting state yields a degenerate density profile. For an isentropic flow with $ P(\rho) = \rho^\gamma, \gamma > 1 $, this implies $ \rho^{\gamma-1} \simeq z $, referred to as the physical vacuum in the study of compressible flows (see, e.g., \cite{Liu1996,Jang2011}). The main characteristics of the physical vacuum is the H\"older continuity of the background density profile, whose derivatives are singular at $ z = 0 $. While there are some recent developments in the global stability of background solutions to compressible Euler or Navier-Stokes equations for one-dimensional or radial-symmetric flows (see, e.g., \cite{LuoXinZeng2016,LuoXinZeng2015,Hadzic2016a,Hadzic2016}), the corresponding multi-dimensional problem is mostly open. On the other hand, after formally filtering out the acoustic waves by sending the Mach number and the Froude number to zero at the same rate in the compressible Navier-Stokes equations with physical vacuum, the resulting equations appear to be the aforementioned anelastic system with $ \rho = z^\alpha, \alpha = 1/(\gamma-1) > 0 $.

In this work, we aim at studying the well-posedness issue of strong solutions to \eqref{eq:2d-anel-NS} in $ \Omega :=2 \mathbb T^{n-1} \times (0,1) = \lbrace \vec{x} = (x,z) \rbrace \subset \mathbb R^{n-1} \times \mathbb R = \mathbb R^n $, where $ n \in \lbrace 2,3\rbrace $ denotes the spatial-dimension. Specifically, we will study system \eqref{eq:2d-anel-NS} with two kinds of density profiles:
\begin{align}
	& \text{non-degenerate case:} ~ \inf_{\Omega} \rho > 0, ~ \text{and} ~ \rho ~ \text{is smooth in $ 2 \mathbb T^n $}; \label{density-profile-smooth} \\
	& \text{physical vacuum case:}~ \rho = \rho_{\mathrm{pv}}:= z^\alpha (2-z)^\alpha ~ \text{for some $ \alpha > 3/2 $}.  \label{density-profile-physical-vacuum}
\end{align}
(See Remark \ref{rm:compatible-condition}, below.)
After denoting the velocity field $ u $ by its horizontal component $ v $ and its vertical component $ w $, i.e., $ u = (v, w)^\top $, where $ v $ is a scaler if $ n = 2 $ and a two-dimensional vector if $ n = 3 $, system \eqref{eq:2d-anel-NS} is complemented with the following stress-free and non-permeable boundary conditions, respectively,
\begin{equation}\label{bc:2d-anel-NS}
    \dz v \big|_{z=0,1}, ~ w \big|_{z=0,1} = 0,
\end{equation}
and initial data
\begin{equation}\label{in:2d-anel-NS}
	\begin{gathered}
	u\big|_{t=0} = u_{in} = (v_{in}, w_{in})^\top \in H^2(\Omega), \\
	\text{and} ~ \supp (v_{in}),\supp(w_{in}) \subset 2 \mathbb T^{n-1} \times (\delta, 1-\delta) ~ \text{for some} ~ \delta \in (0,1/4).
	\end{gathered}
\end{equation}
Compatibility conditions for $ u_{in} $ are given by,
\begin{equation}\label{in:compatibility}
	\begin{gathered}
		\dz v_{in} \big|_{z=0,1}, w_{in} \big|_{z=0,1} = 0, ~ \dv(\rho u_{in}) = 0, \\
		\rho u_{t}\big|_{t=0} =  \rho u_{in,1} := \Delta u_{in} - \rho u_{in} \cdot \nabla u_{in} - \rho \nabla p_{in} \in L^2(\Omega),
	\end{gathered}
\end{equation}
where $ p_{in} $ is the solution to the following elliptic problem
\begin{equation}\label{initial-p}
	\begin{gathered}
	\dv (\rho \nabla p_{in}) = \dv \Delta u_{in} - \dv(\rho u_{in}\cdot \nabla u_{in} ),\\
	 \dz p_{in}\big|_{z=0,1} = 0, ~ \int_\Omega p_{in} \idx = 0.
	\end{gathered}
\end{equation}
Hereafter, for any subscript $ s $, we will alway use the letter $ v_{s} $ to denote the horizontal velocity, $ w_s $ to denote the vertical velocity, and $ u_s $ to denote the velocity field, i.e., $ u_s = (v_s,w_s)^\top $.
\begin{remark}\label{rm:compatible-condition}
	The meaning of \eqref{density-profile-smooth} is that the non-degenerate profile $ \rho $ can be extended as a non-degenerate and smooth function in $ 2 \mathbb T^n $.
	
	On the other hand, one can replace $ \rho_\mathrm{pv} $ in \eqref{density-profile-physical-vacuum} with any density profile which satisfies the aforementioned physical vacuum near $ z = 0 $ (i.e., $ \rho \simeq z^\alpha $ near $ z = 0 $), and is smooth and non-degenerate at $ z = 1 $. The requirement of smoothness and the non-degeneracy of the density profile at $ z =1 $ is owing to technical reasons.
\end{remark}
\begin{remark}\label{rm:compatible-condition-2}
	We require the supports of $ v_{in} $ and $ w_{in} $ to be away from the boundary $ \lbrace z = 0, 1 \rbrace $.
	This can be modified as follows:
	\begin{gather*}
	\text{let} ~ \eta, \zeta ~ \text{be even, odd in the $ z $-variable, respectively, }\\
	\text{and} ~  (\eta, \zeta) \in H^2(2\mathbb T^n;\mathbb R^{n-1})\times H^2(2\mathbb T^n;\mathbb R); \\ \text{then take} ~  (v_{in},w_{in})^\top = (\eta\big|_{\Omega},\zeta\big|_{\Omega})^\top.
	\end{gather*}
	Equation \eqref{initial-p} should be considered as a compatibility condition of $ u_{in} $.
\end{remark}

In comparison to the Navier-Stokes system (see, e.g., \cite{Constantin1988}), the density profile interacts with the velocity field. To explain this statement, let us forget about the boundary and consider \eqref{eq:2d-anel-NS} in $ 2\mathbb T^n $ for a moment. Also we assume the density profile $ \rho $ is non-degenerate and smooth. Let $ u $ be any smooth vector field, and suppose that it can be decomposed, in analogy with the Helmholtz decomposition, as
\begin{equation*}
	u = u_{\rho,\sigma} + \nabla \phi,
\end{equation*}
where $ \dv(\rho u_{\rho, \sigma}) = 0 $, and $ \phi $ is a smooth function. One can see that $ u_{\rho,\sigma} $ and $ \nabla \phi $ are orthogonal in $ L^2(\rho \idx) $, i.e., the square--integrable space with respect to the measure $ \rho \idx  $, where $ \idx $ is the Lebesgue measure. Then one can easily see that, the $ H^s $--regularity of $ u_{\rho, \sigma} $, $ s \geq 1 $, depends not only on the regularity of $ u $, but also on that of $ \rho $. Such an interaction of $ \rho $ in the $ H^s $-regularity estimates causes the main difficulty in the study of strong solutions. As one will see later, one will need to consider the interaction of the pressure term $ \rho \nabla p $ with the nonlinearity term $ \rho u\cdot \nabla u $ as well as the viscosity term $ \Delta u $. We take advantage of the regular density profile in \eqref{density-profile-smooth}, and consider an associated problem in $ 2\mathbb T^n $, whose solutions satisfies \eqref{eq:2d-anel-NS} in $ \Omega $.
Then we employ an elementary approach in the Galerkin's approximation which takes into account the aforementioned interactions. After studying the regularity, we restrict our solutions back to the original domain $ \Omega $ and obtain a unique strong solution to \eqref{eq:2d-anel-NS} with \eqref{bc:2d-anel-NS} in the non-degenerate case, i.e., \eqref{density-profile-smooth}.

To deal with the physical vacuum profile in \eqref{density-profile-physical-vacuum}, we approximate the problem with a sequence of non-degenerate profiles in the class of \eqref{density-profile-smooth}.
The existence theorem in the non-degenerate case yields a sequence of approximating solutions.
Then we derive the necessary uniform weighted estimates. To handle the physical vacuum density profile, the desired strong solutions to \eqref{eq:2d-anel-NS} in the physical vacuum case, i.e., \eqref{density-profile-physical-vacuum}, are constructed as the limit of the approximating sequence. However, the solutions that we obtain lack regularity on the boundary $ \lbrace z = 0 \rbrace $, due to the weighted estimates. In particular, the solutions are not regular enough to have trace of $ \nabla u $ on $ \lbrace z = 0 \rbrace $, which causes troubles when one try to show the uniqueness of solutions. We employ the arguments originated in \cite{Serrin1963} for the Navier-Stokes system to establish the uniqueness of strong solutions.

Next, we
sum up the main theorems. The first theorem concerns the local well-posedness of strong solutions to \eqref{eq:2d-anel-NS}:
\begin{thm}\label{thm:local}
	Let $ \rho $ satisfy either \eqref{density-profile-smooth} or \eqref{density-profile-physical-vacuum}.
	Consider initial data $ u_{in} \in H^2 $ satisfying \eqref{in:2d-anel-NS} and \eqref{in:compatibility}. There exists a unique strong solution $ (u, p) $ to \eqref{eq:2d-anel-NS} with \eqref{bc:2d-anel-NS} in $ [0,T] $, for some $ T\in (0,\infty) $.
	In the case of \eqref{density-profile-smooth}, the strong solution satisfies the regularity:
		\begin{gather*}
			u \in L^\infty (0,T; H^2(\Omega))\cap L^2(0,T;H^3(\Omega)),	\\
			\dt u \in L^\infty (0,T;L^2(\Omega))\cap L^2(0,T;H^1(\Omega)), \\
			\nabla p \in L^\infty (0,T; L^2(\Omega))\cap L^2(0,T;H^1(\Omega)).
		\end{gather*}
		In the case of \eqref{density-profile-physical-vacuum}, the strong solution satisfies the regularity:
		\begin{gather*}
		 u, \nabla_h u \in L^\infty(0,T;H^1(\Omega)), ~ u \in L^2(0,T;H^1(\Omega)), \\
		  \rho_{\mathrm{pv}} \partial_{zz} u \in L^\infty(0,T; L^2(\Omega)), ~ \rho_{\mathrm{pv}}^{1/2} u_t \in L^\infty(0,T;L^2(\Omega)), \\
		  u_t \in L^2(0,T;H^1(\Omega)), ~ \rho_{\mathrm{pv}}^{2} \nabla p \in L^\infty(0,T;L^2(\Omega)).
		\end{gather*}
		See section \ref{sec:preliminary} for the notations $ \partial_{zz}, \nabla_h $ etc..
\end{thm}
We refer the detailed description of local well-posedness to Theorem \ref{prop:local-exist-nonsingular} and Theorem \ref{prop:local-exist-2d-anel-NS}, below.
At the same time, we also have the following theorem concerning global well-posedness of strong solutions:
\begin{thm}\label{thm:global}
	Under either one of the following conditions, the existing time of the local strong solutions constructed in Theorem \ref{thm:local} becomes infinite:
	\begin{itemize}
		\item[1.] $ n = 2 $;
		\item[2.] $ n = 3 $, provided initial velocity $ u_{in} $ satisfies, for $ \rho $ satisfying either \eqref{density-profile-smooth} or \eqref{density-profile-physical-vacuum},
		\begin{equation*}
			\norm{\rho^{1/2}  u_{in}}{\Lnorm{2}}^2 + \norm{\nabla u_{in}}{\Lnorm{2}}^2 + \norm{\rho^{1/2} u_{in,1}}{\Lnorm{2}}^2 \leq \mu^2,
		\end{equation*}
	\end{itemize}
	with some $ \mu \in (0,1) $, small enough.
\end{thm}


Notice that, by taking $ \rho \equiv 1 $, Theorems \ref{thm:local} and \ref{thm:global} apply to the homogeneous Navier--Stokes equations (see, e.g., \cite{Temam1984}). The compatibility conditions in \eqref{in:compatibility} are similar to those in the study of the non-homogeneous incompressible Navier--Stokes equations (see, e.g., \cite{Choe-Kim-2003}).

The rest of this work is organized as follows. In section \ref{sec:preliminary}, we summarize the notations, definitions and inequalities which will be used in this paper.
In section \ref{sec:local-solutions-nonsingular}, we construct the local strong solutions to \eqref{eq:2d-anel-NS} with the non-degenerate density profile. In section \ref{sec:physical-vacuum}, we consider \eqref{eq:2d-anel-NS} with the physical vacuum profile, when the uniform estimates and approximation arguments are presented. These two sections finish the proof of Theorem \ref{thm:local}. In the last two sections, i.e., sections \ref{sec:global-2d} and \ref{sec:global-3d}, we employ some global a priori estimates, which lead to the proof of Theorem \ref{thm:global}.

\section{Preliminaries}\label{sec:preliminary}

Throughout this paper, we use the following definition of strong solutions:
\begin{definition}[Strong solutions]\label{def:strong-solution}
	$ (u,p) $ is called
	a strong solution to \eqref{eq:2d-anel-NS} if system \eqref{eq:2d-anel-NS} holds almost everywhere in $ \Omega $.
\end{definition}

We use the notation $ \partial_x $ to denote the spatial derivative in the horizontal direction, i.e. derivative with respect to $ x \in 2\mathbb T $, when $ n = 2 $, and $ x_1, x_2 $ for $ x = (x_1,x_2)^\top \in 2\mathbb T^2 $, when $ n = 3 $; the notation $ \partial_z $ to denote the spatial derivative in the vertical direction; the notation $ \partial_t $ to denote the temporal derivative; $ u_s $ for $ s \in \lbrace t, x, z \rbrace $ is short for $ \partial_s u $; also $ u_{s_1s_2} $ and $ \partial_{s_1s_2} u $ for $ s_1, s_2 \in \lbrace t, x,z \rbrace $ are short for $  \partial_{s_1} \partial_{s_2} u $. $ \dvh, \nabla_h, \Delta_h $ are used to denote the divergence, the gradient, the Laplace, respectively, in horizontal direction, i.e. \begin{gather*}
	\dvh = \dx, ~~ \text{when} ~ n = 2, ~~\text{and}~~ \dvh =  \nabla_h \cdot, ~~ \text{when} ~ n = 3, \\
	\nabla_h = \dx, ~~ \text{when} ~ n = 2, ~~\text{and}~~ \nabla_h = \biggl( \begin{array}{c} \partial_{x_1} \\ \partial_{x_2} \end{array} \biggr), ~~ \text{when} ~ n = 3,\\
	\Delta_h = \partial_{xx}, ~~ \text{when} ~ n =2, ~~ \text{and} ~~ \Delta_h = \dvh \nabla_h, ~~ \text{when} ~ n = 3.
\end{gather*}
In addition, we abuse the notation:
\begin{gather*}
	\int \cdot \idx = \int_{\Omega} \cdot \idx, ~~ \text{or} ~ \int \cdot \idx = \int_{2\mathbb{T}^n} \cdot \idx, ~~ \text{depending on the context}.
\end{gather*}
$ \Lnorm{p}, \Hnorm{k} $ are used to denote $ \Lnorm{p}(\Omega), \Hnorm{k}(\Omega) $ or $ \Lnorm{p}(2\mathbb T^n), \Hnorm{k}(2\mathbb T^n) $, depending on the context.

Also, we summarize the symmetric-periodic extensions in the following:
\begin{definition}[Symmetric-periodic extensions]\label{def:extension}
	For any smooth function $ f $ defined in $ \Omega := 2 \mathbb{T}^{n-1} \times (0,1) $, one can extend it to an even function in $ \Omega_{\pm} := 2 \mathbb{T}^{n-1} \times (-1,1) $, using the even-symmetric extension $ \mathfrak E_s^+ $, defined by
	\begin{equation}\label{symmetric-extension}
		\mathfrak{E}_s^+ f(x,z) : = \begin{cases} f(x,|z|), & x \in 2 \mathbb T^{n-1}, z \in (-1,1) \backslash \lbrace 0 \rbrace, \\
		\lim\limits_{z\rightarrow 0^+} f(x,z), & x \in 2 \mathbb{T}^{n-1}, z= 0.
		\end{cases}
	\end{equation}
	In addition, if $ \lim\limits_{z\rightarrow 0^+} f(x,z) = 0 $, one can also extend $ f $ to a odd function in $ \Omega_{\pm} 
	$, using the odd-symmetric extension $ \mathfrak E_s^- $, defined by
	\begin{equation}\label{symmetric-extension-odd}
		\mathfrak{E}_s^- f(x,z) : = \begin{cases} \dfrac{z}{|z|}f(x,|z|), & x \in 2 \mathbb T^{n-1}, z \in (-1,1)\backslash\lbrace 0 \rbrace, \\
		0, & x \in 2\mathbb T^{n-1}, z = 0 .\end{cases}
	\end{equation}
	
	Also, for any smooth function $ g $ defined in $ \Omega_{\pm} $, one can extend it to a function in $ 2 \mathbb T^n $ using the periodic extension $ \mathfrak E_p $, defined by
	\begin{equation}\label{periodic-extension}
		\begin{aligned}
		& \mathfrak E_p g(x,z) : = \begin{cases} g(x, z - 2k), & (x,z) \in 2 \mathbb T^{n} \backslash \lbrace z \in 1 + 2 \mathbb Z \rbrace,\\
		\dfrac{1}{2} \bigl( \lim\limits_{z\rightarrow 1^-}g(x,z) + \lim\limits_{z \rightarrow -1^+}g(x,z) \bigr) & x \in 2 \mathbb T^{n-1}, z\in 1 + 2\mathbb Z, \end{cases}\\
		& ~~~~ ~~~~ k ~ \text{is the integer such that $ z - 2k \in (-1,1) $. }\end{aligned}
	\end{equation}
	Then the even-symmetric-periodic extension operator is defined by
	\begin{equation}\label{symmetric-periodic-extension}
		\mathfrak E_{sp}^+ f := \mathfrak E_{p} \mathfrak E_{s}^+ f, ~~~~ f ~ \text{is a function defined in $\Omega$.}
	\end{equation}
	The odd-symmetric-periodic extension operator is defined by
	\begin{equation}\label{symmetric-periodic-extension}
		\mathfrak E_{sp}^- f := \mathfrak E_{p} \mathfrak E_{s}^- f , ~~~~ f ~ \text{is a function defined in $\Omega$ and $ f(0) = 0 $.}
	\end{equation}
\end{definition}

To study \eqref{eq:2d-anel-NS} in the case of physical vacuum, i.e., \eqref{density-profile-physical-vacuum},
we will need to apply the following Hardy-type inequalities:
\begin{lm}[Hardy-type inequalities]\label{lm:hardy-type-ineq} Let $ k \neq -1 $ be a real number. Suppose that a function $ f \in C^1([0,1]) $ satisfies $ \int_0^1  (z+\varepsilon)^{k+2} (|f|^2(z) + |f'|^2(z)) \,dz < \infty $. Then
for some positive constant $ C_k \in (0,\infty) $, independent of $ \varepsilon \in (0,1) $, one has
	\begin{enumerate}
		\item for $ k > - 1 $,
		\begin{equation}\label{hardy-1}
			\int_0^1 (z+\varepsilon)^k |f(z)|^2 \,dz \leq C_k \int_0^1 (z+\varepsilon)^{k+2} (|f(z)|^2 + |f'(z)|^2) \,dz;
		\end{equation}
		\item for $ k < -1 $,
		\begin{equation}\label{hardy-2}
			\int_0^1 (z+\varepsilon)^k |f(z) - f(0)|^2 \,dz \leq C_k \int_0^1 (z+\varepsilon)^{k+2}|f'(z)|^2 \,dz.
		\end{equation}
	\end{enumerate}
	In particular, after taking $ \varepsilon = 0 $ in \eqref{hardy-1} and \eqref{hardy-2}, one will arrive at the standard Hardy's inequalities.
\end{lm}
\begin{proof}
{\bf\noindent Inequality \eqref{hardy-1}: $ k > -1 $.} The mean value theorem guarantees that there is a $ z^* \in [1/2,1] $ such that $ 2 |f(z^*)|^2 \leq \int_{1/2}^1 |f(\xi)|^2 \,d\xi \leq 2^{k+2} \int_{1/2}^1 (\xi+\varepsilon)^{k+2} |f(\xi)|^2 \,d\xi $. Then applying the Fundamental Theorem of Calculus and the Fubini's theorem yields, since $ k + 1 > 0 $,
\begin{align*}
	& \int_0^1 (z+\varepsilon)^k |f(z)|^2 \,dz \lesssim \int_0^1 (z+\varepsilon)^k  \bigl(|\int_{z^*}^z f(\xi) f'(\xi) \,d\xi| + |f(z^*)|^2 \bigr) \,dz\\
	& ~~ \lesssim \int_0^1 (z+\varepsilon)^k \int_{z}^1 |f(\xi)|  |f'(\xi)| \,d\xi  \,dz + \int_0^1(z+\varepsilon)^k \,dz \\
	& ~~~~ \times  \int_{1/2}^1 (\xi+\varepsilon)^{k+2} |f(\xi)|^2 \,d\xi  = \int_0^1 \int_0^{\xi} (z+\varepsilon)^k |f(\xi)| |f'(\xi)|    \,dz \,d\xi \\
	& ~~~~ + \dfrac{1}{k+1}((1+\varepsilon)^{k+1} - \varepsilon^{k+1} ) \int_{1/2}^1 (\xi+\varepsilon)^{k+2} |f(\xi)|^2 \,d\xi \\
	& ~~ \lesssim \int_0^1 (\xi+\varepsilon)^{k+1}|f(\xi)| | f'(\xi)| \,d\xi + \int_0^1 (\xi+\varepsilon)^{k+2} |f(\xi)|^2 \,d\xi \\
	& ~~ \lesssim \delta \int_0^1(\xi+\varepsilon)^k |f(\xi)|^2 \,d\xi + C_\delta \int_0^1 (\xi+\varepsilon)^{k+2} |f'(\xi)|^2 \,d\xi\\
	& ~~~~ + \int_0^1 (\xi+\varepsilon)^{k+2} |f(\xi)|^2 \,d\xi,
\end{align*}
where $ \delta> 0 $ is an arbitrary constant and $ C_\delta = {1}/{\delta} $. Then after choosing $ \delta $ small enough, this finishes the proof of \eqref{hardy-1}.
{\par\bf\noindent Inequality \eqref{hardy-2}: $ k < -1 $.} Without loss of generality, we assume $ f(0) = 0 $. Then, again, the Fundamental Theorem of Calculus implies that $ |f(z)|^2 = \int_0^z 2 f(\xi) f'(\xi) \,d\xi $. Thus, since $ k + 1 < 0 $,
\begin{align*}
	& \int_0^1 (z+ \varepsilon)^k |f(z)|^2 \,dz \lesssim \int_0^1 (z+\varepsilon)^k \int_0^z |f(\xi)||f'(\xi)|\,d\xi \,dz\\
	& ~~~~ = \int_0^1 \int_{\xi}^1 (z+ \varepsilon)^k |f(\xi)||f'(\xi)| \,dz \,d\xi \\
	& ~~~~ \lesssim \int_0^1   (\xi+ \varepsilon)^{k+1} |f(\xi)||f'(\xi)| \,d\xi \lesssim \bigl(\int_0^1 (\xi+\varepsilon)^k |f(\xi)|^2 \,d\xi \bigr)^{1/2} \\
	& ~~~~ \times \bigl(\int_0^1 (\xi+\varepsilon)^{k+2} |f'(\xi)|^2 \,d\xi \bigr)^{1/2}.
\end{align*}
Thus \eqref{hardy-2} follows.
\end{proof}

\section{The non-degenerate case}\label{sec:local-solutions-nonsingular}

Recall that our goal is to construct the local strong solutions to the anelastic Navier-Stokes equations, \eqref{eq:2d-anel-NS}, i.e.,
\begin{equation}\label{eq:nonsingular-anel-NS}
    \begin{cases}
       \rho (\dt u + u \cdot \nabla u) + \rho \nabla p = \Delta u & \text{in} ~ \Omega, \\
        \dv (\rho u) = 0 & \text{in} ~ \Omega,
    \end{cases}
\end{equation}
with \eqref{bc:2d-anel-NS} in the case of non-degenerate background density profiles, i.e., \eqref{density-profile-smooth}. In fact, we will only need $ \inf_{\Omega}\rho > 0 $ and
\begin{equation}\label{smooth-density}
\mathfrak E_{sp}^+ \rho \in C^3(2\mathbb T^n).
\end{equation}
Recall that, $ \Omega = 2\mathbb T^{n-1} \times (0,1) $, $ n = 2,3 $.

Our strategy of constructing solutions is: first, we introduce a problem in $ 2 \mathbb T^{n} $, which is associated with \eqref{eq:nonsingular-anel-NS}; then we construct the solutions with enough regularity to the associated problem; by restricting such solutions to the associated problem back in $ \Omega $, we obtain the required solutions to \eqref{eq:nonsingular-anel-NS}.

The following theorem is the main part of this section:
\begin{thm}\label{prop:local-exist-nonsingular}
Let $ \rho $ be a strict positive scalar function in $ \Omega $ that satisfies \eqref{smooth-density}. Consider initial data $ u_{in} \in H^2 $ satisfying \eqref{in:2d-anel-NS} and \eqref{in:compatibility}. Then there exists a unique strong solution $ (u, p) $ to \eqref{eq:nonsingular-anel-NS}, with \eqref{bc:2d-anel-NS}, in $ [0,T^*] $, for some $ T^* \in(0,\infty) $. Moreover, the strong solution satisfies the following regularity:
\begin{gather*}
	u \in L^\infty (0,T^*; H^2(\Omega))\cap L^2(0,T^*;H^3(\Omega)),	\\
	\dt u \in L^\infty(0,T^*; L^2(\Omega))\cap L^2(0,T^*;H^1(\Omega)),\\
	\nabla p \in L^\infty (0,T^*; L^2(\Omega))\cap L^2(0,T^*;H^1(\Omega)).
\end{gather*}
	 Furthermore, the following estimates hold:
\begin{equation}\label{nonsingular-regularity}
	\begin{aligned}
		& \sup_{0\leq t \leq T^*} \bigl( \norm{u(t)}{\Hnorm{2}}^2 + \norm{u_t(t)}{\Lnorm{2}}^2 + \norm{\nabla p(t)}{\Lnorm{2}}^2 \bigr) \\
		& ~~~~ + \int_0^{T^*} \bigl( \norm{u(t)}{\Hnorm{3}}^2 + \norm{u_t(t)}{\Hnorm{1}}^2 + \norm{\nabla p(t)}{\Hnorm{1}}^2 \bigr)\,dt \leq C_{in,\rho},
	\end{aligned}
\end{equation}
where $ C_{in,\rho} \in (0,\infty)  $ depends only on the initial data $ u_{in} $ and $$ \inf_{\vec{x}\in\Omega}{\rho(\vec{x})}, ~ \norm{\rho}{C^{3}(\overline{\Omega})} \in (0,\infty) .$$
Also, one can choose $ p $ to satisfy
\begin{equation}\label{ellip:p-1}
	\dz p\big|_{z = 0,1} = 0, ~~ \text{and} ~~  \int_\Omega p \idx =0.
\end{equation}
In addition, let $ u_1, u_2 $ be strong solutions with initial data $ u_{1,in}, u_{2,in} $, respectively. Then the following estimate holds,
\begin{equation}\label{stability-est-u}
	\begin{gathered}
	\sup_{0\leq t \leq T^*}\norm{u_{1}(t)-u_2(t)}{\Lnorm{2}}^2 + \int_0^{T^*} \norm{\nabla (u_{1}(t) - u_2(t))}{\Lnorm{2}}^2 \,dt \\
	~~~~ ~~~~ \leq C_{in,1,2,T^*} \norm{u_{in,1} - u_{in,2}}{\Lnorm{2}}^2,
	\end{gathered}
\end{equation}
where $ T^* \in(0,\infty) $ is the co-existence time of the solutions, and $ C_{in,1,2,T^*} \in (0,\infty) $ depends on the initial data and $ T^* $.
\end{thm}

\begin{remark}\label{rm:density-profile}
	We observe that conditions \eqref{in:compatibility} and \eqref{smooth-density} are essential factors in Theorem \ref{prop:local-exist-nonsingular}. See Remark \ref{rm:approximating-density} and Remark \ref{rm:uniqueness}, below.
\end{remark}

In fact, we will only show the proof of Theorem \ref{prop:local-exist-nonsingular} when $ n = 2 $. The case when $ n = 3 $ is similar and we omit it for the sake of clarity of our presentation.
The proof is done in the following steps: introducing the associated problem via the symmetric-periodic extension; introducing the Galerkin approximating problem; establishing existence of strong solutions; improving the regularity; establishing uniqueness and continuous dependency on the initial data.

\vspace{0.25cm}
{\par\noindent\bf Step 0:} the associated problem. We observe that
system \eqref{eq:nonsingular-anel-NS} is invariant with respect to the following symmetry:
\begin{equation}\tag{SYM}
\label{SYM}
\begin{gathered}
\text{$ \rho, v, w, p $ are even, even, odd, even, respectively,} \\
\text{with respect to the $z$-variable.}
\end{gathered}
\end{equation}
Recalling $ \Omega = 2\mathbb T \times (0,1) $,
that is to say,
by extending any solution $ (\rho, v, w, p) $ to system \eqref{eq:nonsingular-anel-NS} to
\begin{equation*}
		\rho_{\pm} : = \mathfrak E_{s}^+ \rho, ~ v_{\pm} : = \mathfrak E_{s}^+ v, ~
		w_{\pm} : = \mathfrak E_{s}^- w, ~ p_{\pm} : = \mathfrak E_{s}^+ p,
\end{equation*}
$ (\rho_\pm, v_\pm, w_\pm, p_\pm) $ satisfies the same equations as in system \eqref{eq:nonsingular-anel-NS} in domain $ \Omega_\pm := 2\mathbb T \times (-1,1) $.

Then the new system for the extended functions $ (\rho_{\pm}, v_{\pm}, w_{\pm}, p_{\pm} ) $
is invariant with respect to translation $ z \leadsto z \pm 2 $. Thus, we further extend $ (\rho_{\pm}, v_{\pm}, w_{\pm}, p_{\pm} ) $ periodically in the $ z $-variable by applying $ \mathfrak E_p $ to $ (\rho_{\pm}, v_{\pm}, w_{\pm}, p_{\pm} ) $. Combining these two extensions together, we obtain,
\begin{equation*}
		\rho^* : = \mathfrak E_{sp}^+ \rho, ~ v^* : = \mathfrak E_{sp}^+ v, ~
		w^* : = \mathfrak E_{sp}^- w, ~ p^* : = \mathfrak E_{sp}^+ p,
\end{equation*}
and $ (\rho^*, v^*, w^*, p^*) $  satisfies the same equations as in system \eqref{eq:nonsingular-anel-NS} in domain $ 2\mathbb T^2 $.

Therefore, we end up with the same set of equations as in system \eqref{eq:nonsingular-anel-NS} in  periodic domain $ 2 \mathbb T^2 $, and for simplicity, the same notations $ \rho, v, w, p $ are used to denote $ \rho^*, v^*, w^*, p^*$. Such a convention will be adopted in the following. Then we have got the following system,
\begin{equation*}\tag{\ref{eq:nonsingular-anel-NS}'} \label{eq:periodic-anel-NS}
    \begin{cases}
       \rho (\dt u + u \cdot \nabla u) + \rho \nabla p = \Delta u & \text{in} ~ 2 \mathbb T^2, \\
        \dv (\rho u) = 0 & \text{in} ~ 2 \mathbb T^2,
    \end{cases}
\end{equation*}
with symmetry \eqref{SYM}.
We adopt initial data $ (\mathfrak E_{sp}^+v_{in}, \mathfrak E_{sp}^-w_{in})^\top $ for \eqref{eq:periodic-anel-NS}, which we will denote by the same notation as the original initial data, i.e.,  $ (v_{in}, w_{in} )^\top $.
Notice, the boundary conditions in \eqref{bc:2d-anel-NS} are automatically implied by symmetry \eqref{SYM}. In the next step, a Galerkin approximating procedure will be used to construct solutions to \eqref{eq:periodic-anel-NS}.

\vspace{0.25cm}
{\par\noindent\bf Step 1:} the Galerkin approximating problem.
Given any non-negative integer $ m $, we consider the finite dimensional space, denoted by $ X_m $ and defined as follows:
\begin{equation}\label{def:finite-dim-space}
	\begin{aligned}
		& X_m:=  \bigl\lbrace (v_m, w_m, p_m) | v_m = \sum_{\mathbf k \in \mathbb Z_m} a^v_{\mathbf{k}} e^{\pi i k_1 x} \cos (\pi k_2 z), \\
		& ~~ w_m =  \sum_{\mathbf k \in \mathbb Z_m} a^w_{\mathbf{k}} e^{\pi i k_1 x} \sin (\pi k_2 z), ~
		p_n = \sum_{\mathbf k \in \mathbb Z_m \backslash \lbrace (0,0) \rbrace } b_{\mathbf{k}} e^{\pi i k_1 x} \cos (\pi k_2 z),
		\\
		&  ~~ \text{with $ a_{\mathbf k}^v, a_{\mathbf k}^w, b_{\mathbf k} $ being complex-valued scalar functions of $ t $ only and}\\
		& ~~ \text{satisfying the reality condition, i.e.,} \\
		& ~~
		 a_{(k_1,k_2)}^v = \overline{a_{(-k_1,k_2)}^v},
		a_{(k_1,k_2)}^w = \overline{a_{(-k_1,k_2)}^w},
		b_{(k_1,k_2)} = \overline{b_{(-k_1,k_2)}}
		\bigr\rbrace, \\
		& \text{where} ~
		\mathbb Z_m
		  : = \lbrace \mathbf k = (k_1,k_2) \in \mathbb Z \times \mathbb Z,  -m \leq k_1 \leq m,
		  0\leq  k_2 \leq m \rbrace
		\bigr\rbrace.
	\end{aligned}
\end{equation}
Notice that, the dimension of $ X_m $ over $ \mathbb R $ is $ 3 (2m+1) (m+1) - 1 $. Also, we define the lower-$m^{th}$-frequency projection operator $ \mathbb P_m $, $ m \geq 0 $, as follows.
\begin{equation}\label{def:projection}
	\begin{aligned}
	& \text{Given} ~ f = \sum_{\mathbf k \in \mathbb Z \times \mathbb Z} c_{\mathbf k} e^{\pi ik_1 x +  \pi i k_2 z}, ~ \text{with $c_{\mathbf k} $ being complex-valued} \\
	&  \text{scalar functions of $ t $}, ~
	 \mathbb P_m f := \sum_{\mathbf k \in \mathbb Z_m^\pm} c_{\mathbf k} e^{\pi ik_1 x +  \pi i k_2 z}, ~ \text{where} \\
	& \mathbb Z_m^\pm
		  : = \lbrace \mathbf k = (k_1,k_2) \in \mathbb Z \times \mathbb Z,  -m \leq k_1 \leq m,
		  -m \leq  k_2 \leq m \rbrace.
	\end{aligned}
\end{equation}
Then $ \mathbb P_m $ projects $ (v, w, p) $ with symmetry \eqref{SYM} into $ X_m $ via $ \mathbb P_m  (v, w, p) = ( \mathbb P_m v, \mathbb P_m w, \mathbb P_m p) $, where we have taken $ \int_\Omega p \idx = 0 $.

 Consider any non-negative integer $ m $ and $ (v_m, w_m, p_m) \in X_m $ with $ a_{\mathbf k}^v, a_{\mathbf k}^w, b_{\mathbf k} $ given as in \eqref{def:finite-dim-space}. To solve the problem \eqref{eq:periodic-anel-NS}, we consider the following system of ODE:
 \begin{equation}\label{eq:ODE-m}
 	\begin{cases}
 		\mathbb P_m \bigl\lbrack \rho (\dt v_m + v_m \dx v_m + w_m \dz v_m )  + \rho \dx p_m\bigr\rbrack = \Delta v_m, \\
 		\mathbb P_m \bigl\lbrack \rho (\dt w_m + v_m \dx w_m + w_m \dz w_m )  + \rho \dz p_m \bigr\rbrack = \Delta w_m, \\
 		\dx \mathbb P_m (\rho v_m) + \dz \mathbb P_m (\rho w_m) = 0.
 	\end{cases}
 \end{equation}
To find a solution $ \lbrace (a_{\mathbf k}^v(t), a_{\mathbf k}^w(t), b_{\mathbf k}(t))_{t \in (0,T^*)} \rbrace $, with $ \mathbf k \in \mathbb Z_m $, for some $ T^* \in (0,\infty) $ to \eqref{eq:ODE-m}, we will need to reformulate \eqref{eq:ODE-m} into a system of dimension $ 3(2m+1)(m+1)- 1 $.
In fact, we claim that $ \lbrace b_{\mathbf k} \rbrace $ can be represented as functions of $ \lbrace (a_{\mathbf k}^v(t), a_{\mathbf k}^w(t)) \rbrace $ by inverting a linear algebraic system of dimension $ (2m+1)(m+1) - 1 $, and one can derive a first-order ODE system for $ \lbrace (a_{\mathbf k}^v(t), a_{\mathbf k}^w(t)) \rbrace $ of dimension $ 2(2m+1)(m+1)  $.

Taking $ \dx $ and $ \dz $ to \subeqref{eq:ODE-m}{1} and \subeqref{eq:ODE-m}{2}, respectively, and summing the results together yield, using \subeqref{eq:ODE-m}{3},
\begin{equation*}
	\begin{aligned}
		& \dx \mathbb P_m (\rho \dx p_m) + \dz \mathbb P_m (\rho \dz p_m) = \dx \Delta v_m + \dz \Delta w_m \\
		& ~~~~ - \dx \mathbb P_m \bigl\lbrack \rho( v_m \dx v_m + w_m \dz v_m) \bigr\rbrack  - \dz \mathbb P_m \bigl\lbrack \rho( v_m \dx w_m + w_m \dz w_m) \bigr\rbrack,
	\end{aligned}
\end{equation*}
which is, due to the even symmetry and the strict positivity of $ \rho $, a non-singular linear system with the unknowns $ \lbrace b_{\mathbf k} \rbrace $ of dimension $ (2m+1)(m+1) - 1 $. Thus after solving for $ \lbrace b_{\mathbf k} \rbrace $, \eqref{eq:ODE-m} can be written as the following $ 3(2m+1)(m+1) - 1 $ dimensional system,
\begin{equation}\label{eq:ODE-m-2}
 	\begin{cases}
 		\mathbb P_m \bigl\lbrack \rho (\dt v_m + v_m \dx v_m + w_m \dz v_m )  + \rho \dx p_m\bigr\rbrack = \Delta v_m, \\
 		\mathbb P_m \bigl\lbrack \rho (\dt w_m + v_m \dx w_m + w_m \dz w_m )  + \rho \dz p_m \bigr\rbrack = \Delta w_m, \\
 		\dx \mathbb P_m (\rho \dx p_m) + \dz \mathbb P_m (\rho \dz p_m) = \dx \Delta v_m + \dz \Delta w_m \\
		~~ - \dx \mathbb P_m \bigl\lbrack \rho( v_m \dx v_m + w_m \dz v_m) \bigr\rbrack  - \dz \mathbb P_m \bigl\lbrack \rho( v_m \dx w_m + w_m \dz w_m) \bigr\rbrack.
 	\end{cases}
 \end{equation}
 In particular,  \subeqref{eq:ODE-m-2}{1} and \subeqref{eq:ODE-m-2}{2} form the $ 2(2m+1)(m+1) $ dimensional ODE system of $ \lbrace (a_{\mathbf k}^v(t), a_{\mathbf k}^w(t)) \rbrace $. We remark that, \subeqref{eq:ODE-m}{3} is preserved by the solutions to \eqref{eq:ODE-m-2} with compatible initial data, since  \eqref{eq:ODE-m-2} implies that $  \dt \bigl\lbrack \dx \mathbb P_m (\rho v_m) + \dz \mathbb P_m (\rho w_m) \bigr\rbrack = 0 $. Also, it is easy to verify, after solving for $ \lbrace b_{\mathbf k} \rbrace_{\mathbf k \in \mathbb Z_{m} \backslash \lbrace (0,0) \rbrace} $ with given $ \lbrace a_{\mathbf k}^v, a_{\mathbf k}^w\rbrace_{\mathbf k \in \mathbb Z_m} $ via \subeqref{eq:ODE-m-2}{3} and substituting the solutions to \subeqref{eq:ODE-m-2}{1} and \subeqref{eq:ODE-m-2}{2}, we will have an ODE system of the form
 \begin{equation*}
 	\dt (a_{\mathbf k}^v, a_{\mathbf k}^w) = \mathcal F_{\mathbf k}((a_{\mathbf l}^v, a_{\mathbf l}^w)_{\mathbf l \in \mathbb Z_m}), ~~ \mathbf k \in \mathbb Z_m,
 \end{equation*}
 with $ \lbrace\mathcal F_{\mathbf k}\rbrace_{\mathbf k \in \mathbb Z_m} $ being locally Lipschitz continuous, in fact quadratic functions, with respect to the arguments. Here we need again that $ \rho $ is strictly positive. Then the existence theorem of ODE systems yields that given initial data $$ (v_m,w_m)^\top\Big|_{t=0} : = \mathbb P_m \lbrack(v_{in}, w_{in})^\top \rbrack - \nabla Q_m, $$
 where $ Q_m = \sum_{\mathbf k \in \mathbb Z_m\slash\lbrace(0,0)\rbrace} q_{\mathbf k} e^{\pi i k_1 x} \cos (\pi k_2 z) $, with $ q_{(k_1,k_2)} = \overline{q_{(-k_1,k_2)}} $, is determined by solving, as above, the non-singular linear algebraic system
 $$ \dv \mathbb P_m( \rho \nabla Q_m) = \dv(\mathbb P_m \lbrack \rho \mathbb P_m \lbrack(v_{in}, w_{in})^\top \rbrack \rbrack), $$
  there exists a solution $ (v_m(t),w_m(t),p_m(t))|_{t \in (0, T^*_m)} \in X_m  $ to system \eqref{eq:ODE-m-2}, or equivalently system \eqref{eq:ODE-m}, for some positive constant $ T^*_m \in (0,\infty) $.

  We remark that, as $ m \rightarrow \infty $, $ Q_m \rightarrow 0 $ in $ H^3 $. In fact, owing to the fact $ \dv(\mathbb P_m \lbrack \rho \mathbb P_m \lbrack(v_{in}, w_{in})^\top \rbrack\rbrack ) = \dv(\mathbb P_m \lbrack\rho \mathbb P_m \lbrack(v_{in}, w_{in})^\top \rbrack\rbrack - \mathbb P_m \lbrack\rho (v_{in}, w_{in})^\top \rbrack ) $, the elliptic estimate yields, as $ m\rightarrow \infty $,
  \begin{equation}\label{approximating-density-April}
  \norm{Q_m}{\Hnorm{3}} \leq \norm{\mathbb P_m (\rho \mathbb P_m \lbrack(v_{in}, w_{in})^\top \rbrack) - \mathbb P_m \lbrack\rho (v_{in}, w_{in})^\top \rbrack }{\Hnorm{2}} \rightarrow 0. \end{equation}
  Hence $ (v_m,w_m)\big|_{t=0} $ is an approximation of $ (v_{in}, w_{in}) $.
  \begin{remark}\label{rm:approximating-density}
  	We remind the reader that the smoothness of $ \rho $ in $ 2 \mathbb T^2 $ (i.e., \eqref{smooth-density}) is essential in showing  \eqref{approximating-density-April}.
  \end{remark}

\vspace{0.25cm}
{\par\noindent\bf Step 2:} existence of strong solutions. In order to pass the limit $ m \rightarrow \infty $ in \eqref{eq:ODE-m} to obtain a solution to \eqref{eq:periodic-anel-NS}, we need to establish that the existence time $ T_m^* $, obtained above, is independent of $ m $. This is done via some uniform-in-$m$ estimates. Let $ T_m^{**} \geq T_m^* $ be the maximal existing time of the solutions $ (v_m, w_m) $. All the estimates below in this step are done in the time interval $ [0,T_m^{**}) $.

After taking the $ L^2 $-inner product of \subeqref{eq:ODE-m}{1} and \subeqref{eq:ODE-m}{2} with $ v_m $ and $ w_m $, respectively, summing up the resulting equations and applying integration by parts yield,
\begin{equation}\label{galerkin:L2}
	\begin{aligned}
	& \dfrac{1}{2} \dfrac{d}{dt}  \int_{2\mathbb T^2} \rho ( |v_m|^2 + |w_m|^2) \idx + \int_{2\mathbb T^2} (|\nabla v_m|^2 + |\nabla w_m|^2 )\idx \\
	& ~~~~ = - \int \rho \bigl( v_m \dx v_m v_m + v_m \dx w_m w_m + w_m \dz v_m v_m + w_m \dz w_m w_m \bigr) \idx \\
	& ~~~~ \lesssim \norm{\nabla v_m, \nabla w_m}{\Lnorm{2}} \norm{ v_m, w_m}{\Hnorm{1}}^2 .
	\end{aligned}
\end{equation}
where we have used \subeqref{eq:ODE-m}{3}.
Next, we take the $ L^2 $-inner product of \subeqref{eq:ODE-m}{1} and \subeqref{eq:ODE-m}{2} with $ \dt v_m $ and $ \dt w_m $, respectively. Similarly, after summing up the resulting equations and applying integration by parts, one will have, since $ \rho $ has uniform upper bound and strictly positive lower bound,
\begin{align*}
	& \dfrac{1}{2} \dfrac{d}{dt} \int_{2\mathbb T^2} ( |\nabla v_m|^2 + |\nabla w_m|^2 ) \idx + \int_{2\mathbb T^2} \rho ( |\dt v_{m}|^2 + |\dt w_m|^2 ) \idx \\
	& ~~~~ = - \int \bigl\lbrack \rho ( v_m \dx v_m + w_m \dz v_m ) \dt v_m + \rho ( v_m \dx w_m + w_m \dz w_m ) \dt w_m \bigr\rbrack \idx \\
	& ~~~~ \lesssim \norm{\dt v_m, \dt w_m}{\Lnorm{2}} \norm{v_m, w_m}{\Lnorm{4}} \norm{ \nabla v_m, \nabla w_m}{\Lnorm{4}}\\
	& ~~~~ \lesssim \norm{\rho \dt v_m, \rho \dt w_m}{\Lnorm{2}} \norm{v_m, w_m}{\Lnorm{2}}^{1/2}\norm{v_m, w_m}{\Hnorm{1}}^{1/2} \norm{ \nabla v_m, \nabla w_m}{\Lnorm{2}}^{1/2} \\
	& ~~~~ ~~~~ \times \norm{ \nabla v_m, \nabla w_m}{\Hnorm{1}}^{1/2}
\end{align*}
where we have applied the two-dimensional Sobolev embedding inequality.
Thus, we have, after applying Young's inequality and \eqref{galerkin:L2},
\begin{equation}\label{galerkin:H1-1}
	\begin{aligned}
	& \dfrac{d}{dt} \norm{ v_m,  w_m}{\Hnorm{1}}^2 + \norm{\dt v_m,\dt w_m}{\Lnorm{2}}^2 + \norm{\nabla v_m, \nabla w_m}{\Lnorm{2}}^2 \\
	& ~~~~ \lesssim  \norm{ v_m,  w_m}{\Hnorm{1}}^3 (\norm{ \nabla v_m, \nabla w_m}{\Hnorm{1}} + 1).
	\end{aligned}
\end{equation}
In order to estimate $ \nabla^2 v_m, \nabla^2 w_m $, we rewrite \subeqref{eq:ODE-m}{1} and \subeqref{eq:ODE-m}{2} in the following pressure-viscosity form:
\begin{equation}\label{eq:galerkin-pressure-viscosity-form}
	\begin{gathered}
		- \Delta v_m + \mathbb P_m (\rho \dx p_m) = - \mathbb P_m \bigl\lbrack \rho (\dt v_m + v_m\dx v_m + w_m \dz v_m) \bigr\rbrack, \\
		- \Delta w_m + \mathbb P_m (\rho \dz p_m) = - \mathbb P_m \bigl\lbrack \rho (\dt w_m + v_m\dx w_m + w_m \dz w_m) \bigr\rbrack,
	\end{gathered}
\end{equation}
which yield
\begin{equation}\label{galerkin:H1-1.1}
	\begin{aligned}
		& \norm{- \Delta v_m + \mathbb P_m (\rho \dx p_m) }{\Lnorm{2}} + \norm{ - \Delta w_m + \mathbb P_m (\rho \dz p_m)}{\Lnorm{2}} \\
		& ~~~~ \lesssim \norm{\dt v_m, \dt w_m}{\Lnorm{2}} + \norm{v_m, w_m}{\Lnorm{4}}\norm{\nabla v_m, \nabla w_m}{\Lnorm{4}} \\
		& ~~~~ \lesssim \norm{\dt v_m, \dt w_m}{\Lnorm{2}} + \norm{v_m, w_m}{\Hnorm{1}} \norm{ \nabla v_m, \nabla w_m}{\Lnorm{2}}^{1/2} \\
		& ~~~~  \times \norm{ \nabla v_m, \nabla w_m}{\Hnorm{1}}^{1/2} \lesssim \norm{\dt v_m, \dt w_m}{\Lnorm{2}} +  \norm{ v_m,  w_m}{\Hnorm{1}}^{3/2} \\
		& ~~~~  \times (  \norm{\nabla v_m, \nabla w_m}{\Lnorm{2}}^{1/2} + \norm{\nabla^2 v_m, \nabla^2 w_m}{\Lnorm{2}}^{1/2}).
	\end{aligned}
\end{equation}
Meanwhile, direct calculations show that
\begin{equation}\label{galerkin:H1-1.2}
\begin{aligned}
	& \norm{- \Delta v_m + \mathbb P_m (\rho \dx p_m) }{\Lnorm{2}}^2 + \norm{- \Delta w_m + \mathbb P_m (\rho \dz p_m) }{\Lnorm{2}}^2 \\
	& ~~~~ = \norm{\nabla^2 v_m, \nabla^2 w_m }{\Lnorm{2}}^2 + \norm{\mathbb P_m (\rho \nabla p_m)
	}{\Lnorm{2}}^2
	\\
	& ~~~~
	- 2 \int_{2\mathbb T^2} ( \rho \Delta v_m  \dx p_m +  \rho \Delta w_m  \dz p_m ) \idx.
\end{aligned}
\end{equation}
Since
\begin{gather*}
\rho \Delta v_m  = \Delta (\rho v_m) - 2 \nabla \rho \cdot \nabla v_m - \Delta \rho v_m, \\
\rho \Delta w_m  = \Delta (\rho w_m) - 2 \nabla \rho \cdot \nabla w_m - \Delta \rho w_m,
\end{gather*}
we have, after applying integration by parts,
\begin{equation*}
	\begin{aligned}
		&  \int_{2\mathbb T^2} ( \rho \Delta v_m  \dx p_m +  \rho \Delta w_m  \dz p_m ) \idx \\
		& ~~ = - \underbrace{\int_{2\mathbb T^2} \lbrack\dx \mathbb P_m (\rho v_m) + \dz \mathbb P_m (\rho w_m) \rbrack \Delta p_m\idx}_{=0} - \int_{2\mathbb T^2} \bigl( 2 \nabla \rho \cdot \nabla v_m \dx p_m \\
		& ~~~~ ~~~~ + \Delta \rho v_m \dx p_m + 2 \nabla \rho \cdot \nabla w_m \dz p_m + \Delta \rho w_m \dz p_m  \bigr) \idx \\
		& ~~ \lesssim  \norm{\nabla v_m, \nabla w_m, v_m, w_m}{\Lnorm{2}} \norm{\nabla p_m}{\Lnorm{2}},
	\end{aligned}
\end{equation*}
where we need $ \rho \in C^2(2\mathbb T^2) $. Therefore, \eqref{galerkin:H1-1.1} and \eqref{galerkin:H1-1.2} imply, 
\begin{equation}\label{galerkin:H1-1.3}
	\begin{aligned}
	& \norm{\nabla^2 v_m, \nabla^2 w_m}{\Lnorm{2}} +\norm{\mathbb P_m(\rho \nabla p_m)}{\Lnorm{2}} \lesssim \norm{\dt v_m, \dt w_m}{\Lnorm{2}} \\
	& ~~~~ + \norm{ v_m,  w_m}{\Hnorm{1}}^{3/2} (\norm{\nabla v_m, \nabla w_m}{\Lnorm{2}}^{1/2} + \norm{\nabla^2 v_m, \nabla^2 w_m}{\Lnorm{2}}^{1/2} )\\
	& ~~~~ +   \norm{ v_m,  w_m}{\Hnorm{1}}^{1/2} \norm{\nabla p_m}{\Lnorm{2}}^{1/2}.
	 \end{aligned}
\end{equation}


On the other hand, taking the $ L^2 $-inner product of \subeqref{eq:ODE-m-2}{3} with $ - p_m $ yields
\begin{align*}
	& \int \rho |\nabla p_m|^2 \idx = \int (\Delta v_m \dx p_m + \Delta w_m \dz p_m) \idx\\
	& ~~~~ - \int \lbrack \rho( v_m\dx v_m + w_m \dz v_m) \dx p_m +  \rho( v_m\dx w_m + w_m \dz w_m) \dz p_m \rbrack \idx \\
	& ~~ \lesssim \norm{\nabla p_m}{\Lnorm{2}} \bigl( \norm{\nabla^2 v_m, \nabla^2 w_m}{\Lnorm{2}} + \norm{v_m, w_m}{\Lnorm{4}} \norm{\nabla v_m, \nabla w_m}{\Lnorm{4}} \bigr)\\
	& ~~ \lesssim  \norm{\nabla p_m}{\Lnorm{2}} \bigl( \norm{\nabla^2 v_m, \nabla^2 w_m}{\Lnorm{2}} + \norm{v_m, w_m}{\Hnorm{1}}^{3/2} \norm{\nabla v_m, \nabla w_m}{\Hnorm{1}}^{1/2} \bigr).
\end{align*}
Therefore, after applying the two-dimensional Sobolev embedding inequality, together with the fact that $ \rho $ is strictly positive, we arrive at
\begin{equation}\label{galerkin:H1-1.4}
	\norm{\nabla p_m}{\Lnorm{2}} \lesssim \norm{\nabla^2 v_m, \nabla^2 w_m}{\Lnorm{2}} + \norm{v_m, w_m}{\Hnorm{1}}^3 + 1.
\end{equation}
Then, \eqref{galerkin:H1-1.3} and \eqref{galerkin:H1-1.4} imply
\begin{equation}\label{galerkin:H1-2}
	\norm{\nabla^2 v_m, \nabla^2 w_m}{\Lnorm{2}} + \norm{\nabla p_m}{\Lnorm{2}} \lesssim \norm{\dt v_m, \dt  w_m}{\Lnorm{2}} + \norm{v_m, w_m}{\Hnorm{1}}^3 + 1.
\end{equation}
Consequently, \eqref{galerkin:H1-1} and \eqref{galerkin:H1-2} yield
\begin{equation}\label{galerkin:H1-3}
	\begin{gathered}
		 \dfrac{d}{dt} \norm{v_m,  w_m}{\Hnorm{1}}^2 + \norm{\nabla v_m, \nabla w_m}{\Hnorm{1}}^2 + \norm{\nabla p_m}{\Lnorm{2}}^2 + \norm{\dt v_m, \dt w_m}{\Lnorm{2}}^2\\
		 \lesssim 1 + \norm{v_m,  w_m}{\Hnorm{1}}^6.
	\end{gathered}
\end{equation}
Thus, \eqref{galerkin:H1-3} implies that, there exists $ T^* \in (0,T_m^{**}) $, independent of $ m $, such that
\begin{equation}\label{galerkin:H1}
	\begin{gathered}
	\sup_{0\leq t\leq T^*} \norm{v_m(t), w_m(t)}{\Hnorm{1}}^2 + \int_0^{T^*} \bigl(\norm{v_m(t), w_m(t)}{\Hnorm{2}}^2\\
	 ~~~~ ~~~~ + \norm{\dt v_m(t), \dt w_m(t)}{\Lnorm{2}}^2
	 + \norm{p_m(t)}{\Hnorm{1}}^2  \bigr) \,dt \leq C_{in},
	\end{gathered}
\end{equation}
where $ C_{in} \in (0,\infty) $ depends only on the initial data and $ \inf_{\vec{x}\in 2\mathbb T} \rho (\vec{x}), \norm{\rho}{C^2(2\mathbb T^2)} $, and $ T^* $ is independent of $ m $. Then, after passing $ m \rightarrow \infty $ with a suitable subsequence according to the weak compactness theorem of Sobolev spaces and Aubin's compactness theorem (see, e.g., \cite{Temam1984}), we have obtianed
\begin{equation}\begin{gathered}\label{galerkin:regularity}
	(v, w) \in L^\infty(0,T^*;H^1)\cap L^2(0,T^*;H^2), \\
	(\dt v, \dt w) \in L^2(0,T^*;L^2) , ~ p \in L^2(0,T^*;H^1)
\end{gathered}\end{equation}
such that
\begin{align*}
	(v_m, w_m) & \rightarrow  (v, w), & & \text{in}  ~~ L^\infty(0,T^*;H^{1})\\
	(v_m, w_m) & \rightharpoonup  (v, w), & & \text{weakly in}  ~~ L^2(0,T^*;H^{2}),\\
	(\dt v_m, \dt w_m) & \rightharpoonup  (\dt v, \dt w), & & \text{weakly in}  ~~ L^2(0,T^*;L^{2}),\\
	p_m & \rightharpoonup  p, & & \text{weakly in}  ~~ L^2(0,T^*;H^{1}).\\
\end{align*}
Thus it is easy to verify that $ (u = (v, w), p) $ is a strong solution to \eqref{eq:periodic-anel-NS} with \eqref{galerkin:regularity}, which satisfies,
according to \eqref{galerkin:H1},
\begin{equation}\label{periodic-est-H1}
	\sup_{0\leq t\leq T^*} \norm{u(t)}{\Hnorm{1}}^2 + \int_0^{T^*} \bigl(\norm{u(t)}{\Hnorm{2}}^2 + \norm{\dt u(t)}{\Lnorm{2}}^2
	  + \norm{\nabla p(t)}{\Lnorm{2}}^2  \bigr) \,dt \leq C_{in}.
\end{equation}

\vspace{0.25cm}
{\par\noindent\bf Step 3:} improving the regularity. In this step, we establish the regularity of solution $ (u, p) $ to \eqref{eq:periodic-anel-NS} via some {\it a priori} estimates. In the following, we use $ C_{in, \rho}\in (0,\infty) $ to denote a generic constant depending only on the initial data and on $$ \inf_{\vec{x}\in 2\mathbb T^2}{\rho(\vec{x})}, ~ \sup_{\vec{x}\in 2\mathbb T^2}{\rho(\vec{x})}, ~ \norm{\rho}{C^{3}(2\mathbb T^2)} \in (0,\infty) .$$
Here we focus with our estimates over $ [0,T^*] $. We emphasize that all the estimates in this step are formal and can be proved rigorously via the Galerkin method.

First, we obtain the time-derivative estimate. After applying a time derivative to \subeqref{eq:periodic-anel-NS}{1}, the resulting equation is
\begin{equation}\label{eq:periodic-t}
	\rho (\dt u_t + u \cdot \nabla u_t + u_t \cdot \nabla u) + \rho \nabla p_t = \Delta u_t.
\end{equation}
Then after taking the $ L^2 $-inner product of \eqref{eq:periodic-t} with $ u_t $, one has
\begin{equation}\label{ene-periodic-t}
	\dfrac{1}{2} \dfrac{d}{dt} \int \rho |u_t|^2 \idx  + \int | \nabla u_t|^2 \idx = - \int \rho (u_t \cdot \nabla ) u \cdot u_t \idx.
\end{equation}
The right-hand side of \eqref{ene-periodic-t} can be estimated as follows:
\begin{equation}\label{ene-periodic-t-2}
\begin{aligned}
	& - \int \rho (u_t \cdot \nabla ) u \cdot u_t \idx = \int \rho (u_t \cdot \nabla ) u_t \cdot u \idx \\
	& ~~~~ \lesssim \norm{u_t}{\Lnorm{4}} \norm{u}{\Lnorm{4}}\norm{\nabla u_t}{\Lnorm{2}} \lesssim \norm{u}{\Hnorm{1}}\norm{u_t}{\Lnorm{2}}^{1/2} \norm{u_t}{\Hnorm{1}}^{1/2} \norm{\nabla u_t}{\Lnorm{2}}.
\end{aligned}
\end{equation}
Together with Young's inequality, \eqref{ene-periodic-t} and \eqref{ene-periodic-t-2} imply
\begin{equation*}
	\dfrac{d}{dt} \norm{\rho^{1/2} u_t}{\Lnorm{2}}^2 + \norm{\nabla u_t}{\Lnorm{2}}^2 \lesssim   (\norm{u}{\Hnorm{1}}^2 + \norm{u}{\Hnorm{1}}^4)\norm{u_t}{\Lnorm{2}}^2.
\end{equation*}
Thus applying Gr\"onwall's inequality to the above yields, together with \eqref{periodic-est-H1},
\begin{equation}\label{nonsingular-evol-est}
	\begin{gathered}
	\sup_{0\leq t\leq T^*} \bigl( \norm{u(t)}{\Hnorm{1}}^2 + \norm{u_t(t)}{\Lnorm{2}}^2 \bigr) + \int_0^{T^*} \bigl( \norm{u(t)}{\Hnorm{2}}^2 + \norm{u_t(t)}{\Hnorm{1}}^2 \\
	~~~~ ~~~~ + \norm{\nabla p(t)}{\Lnorm{2}}^2 \bigr) \,dt \leq C_{in,\rho},
	\end{gathered}
	\end{equation}
	where $ T^* $ is given in step 2. Then, following similar arguments as in \eqref{galerkin:H1-2}, one can obtain,
\begin{equation}\label{nonsingular-ellip-est-1}
	\sup_{0\leq t \leq T^*} \bigl( \norm{\nabla^2 u(t)}{\Lnorm{2}} + \norm{\nabla p(t)}{\Lnorm{2}} \bigr) \leq C_{in,\rho}.
\end{equation}

Next, we will sketch the $ H^3 $ estimate of $ u $. First, applying $ \dx $ to \subeqref{eq:periodic-anel-NS}{1} yields,
\begin{equation}\label{eq:nonsingular-dx}
	\dx ( \rho (\dt u + u \cdot \nabla u)) + \rho \nabla \dx p + \dx \rho \nabla p = \Delta \dx u.
\end{equation}
After taking the $ L^2 $-inner product of \eqref{eq:nonsingular-dx} with $ u_{xxx} $ and applying integration by parts, we obtain
\begin{equation}\label{nonsingular-ellip-est-1.1}
		\norm{\nabla \partial_{xx} u}{\Lnorm{2}} \leq C_{in,\rho} \bigl( 1 + \norm{\nabla \dx p}{\Lnorm{2}} + \norm{\dt u}{\Hnorm{1}} + \norm{\nabla^3 u}{\Lnorm{2}}^{1/2} \bigr),
\end{equation}
where we have applied \eqref{nonsingular-evol-est}, \eqref{nonsingular-ellip-est-1}, H\"older's and the Sobolev embedding inequalities. Then, after noticing $ \Delta \dx u = \partial_{xxx} u + \partial_{xzz} u $ and using \eqref{eq:nonsingular-dx}, \eqref{nonsingular-ellip-est-1.1} implies
\begin{equation}\label{nonsingular-ellip-est-1.2}
		\norm{\nabla^2 \partial_{x} u}{\Lnorm{2}} \leq C_{in,\rho} \bigl( 1 + \norm{\nabla \dx p}{\Lnorm{2}} + \norm{\dt u}{\Hnorm{1}} + \norm{\nabla^3 u}{\Lnorm{2}}^{1/2} \bigr).
\end{equation}
Moreover, since
\begin{equation*}
	\partial_{zzz} u = \Delta \partial_z u - \partial_{xxz} u = \dz ( \rho (\dt u + u \cdot \nabla u)) + \rho \nabla \dz p + \dz \rho \nabla p - \partial_{xxz} u,
\end{equation*}
one can also derive
\begin{equation}\label{nonsingular-elllip-est-1.3}
	\norm{\nabla^3 u}{\Lnorm{2}} \leq C_{in,\rho} \bigl( 1 + \norm{\nabla^2 p}{\Lnorm{2}} + \norm{\dt u}{\Hnorm{1}} + \norm{\nabla^3 u}{\Lnorm{2}}^{1/2} \bigr).
\end{equation}
What is left is to obtain the estimate of $ \norm{\nabla^2 p}{\Lnorm{2}} $, or equivalently, the estimate of $ \norm{\Delta p}{\Lnorm{2}} $. We rewrite \subeqref{eq:periodic-anel-NS}{1}, after multiplying it with $ \rho $ and applying $ \dv $ to the resulting, as,
	\begin{equation}\label{nonsingular-ellip-p}
		\begin{aligned}
		& \dv (\rho^2 \nabla p) = \dv \bigl(\rho \Delta u - \rho^2 (\dt u + u\cdot\nabla u) \bigr) \\
		& ~~~~ = - 2 \nabla^2 \log \rho : \nabla (\rho u) + \dv \bigl(2 |\nabla\log \rho|^2 \rho u - \Delta \rho u \bigr) \\
		& ~~~~ ~~~~ - \dv\bigl( \rho^2 (\dt u + u\cdot\nabla u)  \bigr),
		\end{aligned}
	\end{equation}
where we have used \subeqref{eq:periodic-anel-NS}{2} and the identity
\begin{align*}
	& \rho \Delta u = \Delta (\rho u) - 2 \nabla \rho \cdot \nabla u - \Delta \rho u \\
	& ~~ = \Delta (\rho u)  - 2 \nabla \log \rho \cdot \nabla ( \rho u) + 2 |\nabla\log \rho|^2 \rho u - \Delta \rho u.
\end{align*}
Thus we have
\begin{equation*}
		\begin{aligned}
		& \rho^2 \Delta p =
		- \nabla \rho^2 \cdot \nabla p - 2 \nabla^2 \log \rho : \nabla (\rho u) + \dv \bigl(2 |\nabla\log \rho|^2 \rho u - \Delta \rho u \bigr) \\
		& ~~~~ ~~~~ - \dv\bigl( \rho^2 (\dt u + u\cdot\nabla u)  \bigr),
		\end{aligned}
\end{equation*}
which yields, together with \eqref{nonsingular-evol-est} and \eqref{nonsingular-ellip-est-1},
\begin{equation}\label{nonsingular-ellip-est-1.4}
	\norm{\nabla^2 p}{\Lnorm{2}} \leq \norm{\Delta p}{\Lnorm{2}} \leq C_{in,\rho} ( 1 + \norm{\dt u}{\Hnorm{1}} + \norm{\nabla^3 u}{\Lnorm{2}}^{1/2} ).
\end{equation}
Then \eqref{nonsingular-evol-est}, \eqref{nonsingular-elllip-est-1.3} and \eqref{nonsingular-ellip-est-1.4} yield
\begin{equation}\label{nonsingular-ellip-est-2}
	\int_0^{T^*} ( \norm{\nabla^3 u(t)}{\Lnorm{2}}^2 + \norm{\nabla^2 p}{\Lnorm{2}}^2 ) \,dt \leq C_{in, \rho}.
\end{equation}

Now we can restrict the solution in $ \Omega $. It is easy to verify $ (u|_{\Omega},p|_{\Omega} ) $ is the strong solution to \eqref{eq:nonsingular-anel-NS} with the boundary condition \eqref{bc:2d-anel-NS}, and it satisfies \eqref{nonsingular-regularity} owing to \eqref{nonsingular-evol-est}, \eqref{nonsingular-ellip-est-1}, and \eqref{nonsingular-ellip-est-2}. In particular, the regularity of $ p $ allows us to take the trace of $ \nabla p $ on the boundary, and it follows from the construction that $ p $ satisfies \eqref{ellip:p-1}.

\vspace{0.25cm}
{\par\noindent\bf Step 4:} uniqueness and continuous dependency on the initial data. Let $ (u_1, p_1) $, $ (u_2, p_2) $ be two strong solutions to \eqref{eq:nonsingular-anel-NS} with initial data $ u_{in,1},u_{in,2} $, respectively, and $ (u_1, p_1), (u_2, p_2) $ satisfy the estimates in \eqref{nonsingular-regularity} for $ T_1^*, T_2^* \in (0,\infty) $, respectively. That is
\begin{align*}
	& \sup_{0\leq t \leq T^*_1} \bigl( \norm{u_1(t)}{\Hnorm{2}}^2 + \norm{u_{1,t}(t)}{\Lnorm{2}}^2 + \norm{\nabla p_1(t)}{\Lnorm{2}}^2 \bigr) \\
	& ~~~~ + \int_0^{T_1^*} \bigl( \norm{u_1(t)}{\Hnorm{3}}^2 + \norm{u_{1,t}(t)}{\Hnorm{1}}^2 + \norm{\nabla p_1(t)}{\Hnorm{1}}^2 \bigr)\,dt \leq C_{in,1},\\
	& \sup_{0\leq t \leq T^*_2} \bigl( \norm{u_2(t)}{\Hnorm{2}}^2 + \norm{u_{2,t}(t)}{\Lnorm{2}}^2 + \norm{\nabla p_2(t)}{\Lnorm{2}}^2 \bigr) \\
	& ~~~~ + \int_0^{T_2^*} \bigl( \norm{u_2(t)}{\Hnorm{3}}^2 + \norm{u_{2,t}(t)}{\Hnorm{1}}^2 + \norm{\nabla p_2(t)}{\Hnorm{1}}^2 \bigr)\,dt \leq C_{in,2},
\end{align*}
where $ C_{in,1}, C_{in,2} $ depend only on the initial data and
$$ \inf_{\vec{x}\in\Omega}{\rho(\vec{x})}, ~ \norm{\rho}{C^{3}(\overline{\Omega})} \in (0,\infty) .$$
In the following, we denote $ T^* := \min\lbrace T_1^*, T_2^* \rbrace $ and $ C_{in,1,2} := \max\lbrace C_{in,1} , C_{in,2} \rbrace $. Also, let $ u_{12}: = u_1 - u_2, p_{12} := p_1 - p_2 $. Then $ (u_{12}, p_{12}) $ satisfies
\begin{equation}\label{eq:nonsingular-conti}
	\begin{cases}
		\rho (\dt u_{12} + u_1 \cdot \nabla u_{12} + u_{12} \cdot \nabla u_2 ) + \rho \nabla p_{12} = \Delta u_{12} & \text{in} ~ \Omega,\\
		\dv(\rho u_{12}) = 0 & \text{in} ~ \Omega.
	\end{cases}
\end{equation}
Then taking the $ L^2 $-inner product of \subeqref{eq:nonsingular-conti}{1} with $ u_{12} $ yields
\begin{align*}
	& \dfrac{1}{2} \dfrac{d}{dt} \norm{\rho^{1/2} u_{12}}{\Lnorm{2}}^2 + \norm{\nabla u_{12}}{\Lnorm{2}}^2 = - \int (\rho u_{12} \cdot \nabla) u_{2} \cdot u_{12} \idx \\
	& ~~~~ \leq C \norm{\nabla u_2}{\Lnorm{2}} \norm{u_{12}}{\Lnorm{4}}^2 \leq C \norm{\nabla u_2}{\Lnorm{2}} \norm{\rho^{1/2}u_{12}}{\Lnorm{2}}\norm{u_{12}}{\Hnorm{1}}\\
	& ~~~~ \leq  \dfrac{1}{2} \norm{\nabla u_{12}}{\Lnorm{2}}^2 + C (1 + \norm{\nabla u_2}{\Lnorm{2}}^2) \norm{\rho^{1/2} u_{12}}{\Lnorm{2}}^2.
\end{align*}
Then applying Gr\"onwall's inequality yields \eqref{stability-est-u}.

In particular, for $ u_{in,1} = u_{in,2} $, we have $ u_1 \equiv u_2 $ and $ T_1^* = T^*_2 $.

This finishes the proof of Theorem \ref{prop:local-exist-nonsingular} in the case when $ n = 2 $. The case when $ n = 3 $ follows by similar arguments, employing the three-dimensional Sobolev embedding inequalities.

\begin{remark}\label{rm:uniqueness}
	It is worth stressing that thanks to
	the uniqueness of solutions in Step 4, all strong solutions to \eqref{eq:nonsingular-anel-NS} with \eqref{bc:2d-anel-NS} with the same initial data as described in Theorem \ref{prop:local-exist-nonsingular} should be equal to the one constructed by our extension-restriction techniques through Step 0 to Step 3.
\end{remark}

\section{The physical vacuum profile}\label{sec:physical-vacuum}

This section will discuss the anelastic equations \eqref{eq:2d-anel-NS} with \eqref{bc:2d-anel-NS} in the case of  physical vacuum density profile, i.e., \eqref{density-profile-smooth}. We remind the reader that
\begin{equation}\label{density:physical-vacuum}
	\rho_\mathrm{pv}(z):= \bigl\lbrack z(2-z)\bigr\rbrack^\alpha, ~~~~ z \in (0,1).
\end{equation}
Thus in this section, we will study the following system,
\begin{equation}\label{eq:physicalvacuum-anel-NS}
    \begin{cases}
        \rho_\mathrm{pv} (\dt u + u \cdot \nabla u) + \rho_\mathrm{pv} \nabla p = \Delta u & \text{in} ~ \Omega, \\
        \dv (\rho_\mathrm{pv} u) = 0 & \text{in} ~ \Omega,
    \end{cases}
\end{equation}
and we will show the following theorem:
\begin{thm}\label{prop:local-exist-2d-anel-NS}
	Consider $ \alpha > 3/2 $, and initial data $ u_{in} $ as in \eqref{in:2d-anel-NS}, satisfying the compatibility condition \eqref{in:compatibility}, with $ \rho $ replaced by $ \rho_\mathrm{pv} $. There exists a positive constant $ T \in (0,\infty) $ and a unique  strong solution $ (u,p) $ to the anelastic equations \eqref{eq:physicalvacuum-anel-NS} with the boundary condition \eqref{bc:2d-anel-NS} in $ [0,T] $, which satisfies the following regularity:
	\begin{gather*}
		 u, \nabla_h u \in L^\infty(0,T;H^1(\Omega)), ~ u \in L^2(0,T;H^1(\Omega)), \\
		  \rho_\mathrm{pv} \partial_{zz} u \in L^\infty(0,T; L^2(\Omega)), ~ \rho_\mathrm{pv}^{1/2} u_t \in L^\infty(0,T;L^2(\Omega)), \\
		  u_t \in L^2(0,T;H^1(\Omega)), ~ \rho_\mathrm{pv}^{2} \nabla p \in L^\infty(0,T;L^2(\Omega)).
	\end{gather*}
	In addition, we have the estimates
	\begin{equation}\label{total-estimate-physical-vacuum}
	\begin{aligned}
	& \sup_{0\leq t\leq T} \bigl( \norm{u(t)}{\Hnorm{1}}^2 + \norm{\nabla \nabla_h u(t)}{\Lnorm{2}}^2 + \norm{\rho_\mathrm{pv}\partial_{zz} u(t)}{\Lnorm{2}}^2
	 + \norm{\rho_\mathrm{pv}^{1/2} u_{t}(t)}{\Lnorm{2}} \\
	& ~~~~
	 + \norm{\rho_\mathrm{pv}^{2} \nabla p(t)}{\Lnorm{2}} \bigr)
	+ \int_0^T \bigl( \norm{\nabla u(t)}{\Lnorm{2}}^2 + \norm{u_{t}(t)}{\Hnorm{1}}^2\bigr) \,dt \leq C_{in},
	\end{aligned}
	\end{equation}
	where $ C_{in} $ depends only on the initial data.
	Suppose that there are two solutions $ u_1, u_2 $ with initial data $ u_{in,1}, u_{in,2} $ in time interval $ [0,T] $. Then
	\begin{equation*}
	\begin{aligned}
	& \sup_{0\leq t\leq T}\norm{\rho_\mathrm{pv}^{1/2} (u_1(t) - u_2(t)) }{\Lnorm{2}}^2 + \int_0^T \norm{\nabla (u_1(s) - u_2(s))}{\Lnorm{2}}^2 \,ds \\
	& ~~~~ \leq C_{in,T} \norm{\rho_\mathrm{pv}^{1/2} (u_{in,1} - u_{in,2}) }{\Lnorm{2}}^2,
	\end{aligned}
\end{equation*}
for some constant $ C_{in,T}\in (0,\infty) $ depending on initial data and $ T $.
\end{thm}

Apparently, $ \rho_\mathrm{py} $ is well-defined for $ z \in (0,2) $ and it is smooth at $ z = 1 $, and $ \mathfrak E_{sp}^+ \rho_{\mathrm{pv}} $ is smooth except $ z \in 2 \mathbb Z $.  We refer to Remark \ref{rm:compatible-condition} about $ \rho_{\mathrm{pv}} $.

Our strategy is to apply Theorem \ref{prop:local-exist-nonsingular} to obtain a sequence of approximating solutions. To do this, we first have to construct an approximating sequence of $ \rho_\mathrm{pv} $. We start with a lemma.

\begin{lm}\label{lm:approximating-z}
	For any fixed $ \varepsilon \in (0,1) $,
	there exists a function $ q_\varepsilon :[0,1] \mapsto \mathbb R^+ $, which satisfies the following properties:
	\begin{enumerate}
		\item $ \lim\limits_{\varepsilon\rightarrow 0^+} q_\varepsilon (z) =  z(2-z), ~ \forall \, z \in (0,1) $, and the convergence is uniform;
		\item $ \mathfrak E_{sp}^+ q_\varepsilon \in C^3(2\mathbb T^n) $;
		\item $ q_\varepsilon $ is non-decreasing for $ z \in [0,1] $;
		\item $ \dfrac{z+\varepsilon}{4} \leq q_\varepsilon \leq 2(z+\varepsilon) $ for $ z \in [0,1] $;
		\item $ |q_\varepsilon'| + |q_\varepsilon q_\varepsilon''| + |q_\varepsilon^2 q_\varepsilon'''| \leq C $, for all $ z \in [0,1] $, for some constant $ C \in (0,\infty) $, which is independent of $ \varepsilon $.
	\end{enumerate}
\end{lm}
\begin{proof}
	Let $ q_\varepsilon $ be a $ C^\infty([0,1]) $ nondecreasing function satisfying,
	\begin{equation*}
		q_\varepsilon(z) =
			\begin{cases}
			z(2-z) & z \in [\dfrac{\varepsilon}{2}  , 1],\\
			\dfrac{\varepsilon}{2} & z \in [0,  \dfrac{\varepsilon}{4} ].
			\end{cases}
	\end{equation*}
	Then $ \dfrac{\varepsilon}{4} \leq q_\varepsilon(\dfrac{\varepsilon}{2})- q_\varepsilon(\dfrac{\varepsilon}{4}) \leq \varepsilon $, and
	\begin{equation*}
		q_\varepsilon'(z) = \begin{cases}
			2 (1-z) & z \in (\dfrac{\varepsilon}{2},1), \\
			0 & z \in (0,\dfrac{\varepsilon}{4} ),
		\end{cases}
		~~
		q_\varepsilon''(z) = \begin{cases}
			-2 & z \in (\dfrac{\varepsilon}{2},1), \\
			0 & z \in (0,\dfrac{\varepsilon}{4} ),
		\end{cases}
	\end{equation*}
	and $ \lim\limits_{z\rightarrow 1^-} q_\varepsilon^{(k)}(z)  = 0, ~ k \in \lbrace 1,3,4, \ldots \rbrace $. Then one can choose the values of $ q_{\varepsilon}(z) $ for $ z \in (\dfrac{\varepsilon}{4}, \dfrac{\varepsilon}{2}) $ such that $ q_\varepsilon $ satisfies Properties 1--5.
\end{proof}
\begin{remark}
	Property 4 in Lemma \ref{lm:approximating-z} implies that Hardy's inequalities in Lemma \ref{lm:hardy-type-ineq} can be applied with $ z+\varepsilon $ replaced by $ q_\varepsilon $.
\end{remark}

We take $ \rho_{\varepsilon} := q_\varepsilon^\alpha $. Then $ \lbrace \rho_{\varepsilon} \rbrace_{\varepsilon \in (0,1)} $ is an approximating sequence of $ \rho_\mathrm{pv} $. In addition, for any fixed $ \varepsilon \in (0,1) $, $ \rho_\varepsilon $ satisfies \eqref{smooth-density} and $ \inf_{\Omega}\rho_\varepsilon \geq (\dfrac{\varepsilon}{4})^\alpha > 0 $. We also choose initial data $ \lbrace u_{\varepsilon,in} \rbrace_{\varepsilon \in (0,1)} $ such that $ u_{\varepsilon,in} \rightarrow u_{in} $ in $ H^2 $, $ \dv(\rho_\varepsilon u_{\varepsilon,in}) = 0 $, and the compatibility conditions in \eqref{in:compatibility} hold with $ \rho, u_{in} $ replaced by $ \rho_{\varepsilon}, u_{\varepsilon, in} $.

We recall that $ u_{in} $ satisfies the compact support condition in \eqref{in:2d-anel-NS}. Thus for $ \varepsilon \in (0,1) $ small enough and $ q_\varepsilon $ as constructed in Lemma \ref{lm:approximating-z}, one can simply take $ u_{\varepsilon, in} = u_{in} $.

Then applying Theorem \ref{prop:local-exist-nonsingular}, we obtain a sequence of solutions to \eqref{eq:nonsingular-anel-NS} with $ \rho = \rho_{\varepsilon} $, which is denoted as $ \lbrace (u_\varepsilon, p_\varepsilon) \rbrace_{\varepsilon \in (0,1)} $. That is, for some $ T_\varepsilon^* \in (0,\infty) $,
\begin{equation}\label{eq:approximating}
	 \begin{cases}
        \rho_{\varepsilon} (\dt u_\varepsilon + u_\varepsilon \cdot \nabla u_\varepsilon ) + \rho_{\varepsilon} \nabla p_\varepsilon = \Delta u_\varepsilon & \text{in} ~ \Omega, \\
        \dv (\rho_{\varepsilon} u_\varepsilon) = 0 & \text{in} ~ \Omega,
    \end{cases}
\end{equation}
where $ \rho_\varepsilon = q_\varepsilon^\alpha $,
\begin{gather*}
u_\varepsilon \in L^\infty (0,T_\varepsilon^*; H^2(\Omega))\cap L^2(0,T_\varepsilon^*;H^3(\Omega)), \\
\dt u_\varepsilon \in L^\infty (0,T_\varepsilon^*;L^2(\Omega))\cap L^2(0,T^*_\varepsilon; H^1(\Omega)), \\
\nabla p_\varepsilon \in L^\infty (0,T_\varepsilon^*; H^1(\Omega))\cap L^2(0,T_\varepsilon^*;H^2(\Omega)).
\end{gather*}
The boundary condition \eqref{bc:2d-anel-NS} is satisfied with $ u $ replaced by $ u_\varepsilon $, and $ p_\varepsilon $ is chosen such that
\begin{equation}\label{pressure:bc-approximating}
	\dz p_\varepsilon \big|_{z = 0,1} = 0, ~~ \text{and} ~~  \int_\Omega p_\varepsilon \idx =0.
\end{equation}
After passing the limit $ \varepsilon \rightarrow 0^+ $, we will obtain a solution to \eqref{eq:physicalvacuum-anel-NS}.
We establish the required uniform estimates in the following two lemmas.

\begin{lm}\label{lm:local-evol-est}
	For any fixed $ \varepsilon \in (0,1) $, assume that $ \alpha > 1 $ and $ (u_\varepsilon, p_\varepsilon) $ is the solution to \eqref{eq:approximating} as mentioned above. There exists a constant $ T \in (0,\infty) $ independent of $ \varepsilon $, such that the following estimates hold:
	\begin{equation}\label{lene:evo-01}
	\begin{aligned}
	& \sup_{0\leq t\leq T} \bigl( \norm{q_\varepsilon^{\alpha/2}u_\varepsilon (t)}{\Lnorm{2}}^2 + \norm{\nabla u_\varepsilon (t)}{\Lnorm{2}}^2 + \norm{q_\varepsilon^{\alpha/2} u_{\varepsilon,t}(t)}{\Lnorm{2}}^2 \bigr) \\
	& ~~~~  + \int_0^T \bigl( \norm{\nabla u_\varepsilon(t)}{\Lnorm{2}}^2 + \norm{q_\varepsilon^{\alpha/2}u_{\varepsilon,t}(t)}{\Lnorm{2}}^2 + \norm{\nabla u_{\varepsilon,t}(t)}{\Lnorm{2}}^2 \bigr) \,dt \leq C_{in},
	\end{aligned}
	\end{equation}
	where $ C_{in} $ is a constant depending only on initial data.
\end{lm}

\begin{proof}
	Taking the $ L^2 $-inner product of \subeqref{eq:approximating}{1} with $ u_\varepsilon $ implies, after substituting \subeqref{eq:approximating}{2} and \eqref{bc:2d-anel-NS},
\begin{equation}\label{lene:000}
    \dfrac{1}{2} \dfrac{d}{dt} \int q_\varepsilon^\alpha |u_\varepsilon|^2 \idx + \int |\nabla u_\varepsilon|^2 \idx = 0.
\end{equation}
In the meantime, the $ L^2 $-inner product of \subeqref{eq:approximating}{1} with $ u_{\varepsilon,t} $ implies, similarly,
\begin{equation}\label{lene:001}
    \dfrac{1}{2} \dfrac{d}{dt} \int |\nabla u_\varepsilon|^2 \idx + \int q_\varepsilon^\alpha |u_{\varepsilon,t}|^2 \idx = - \int q_\varepsilon^\alpha u_\varepsilon \cdot \nabla u_\varepsilon \cdot u_{\varepsilon, t} \idx.
\end{equation}
The right-hand side of \eqref{lene:001} can be estimated as follows,
\begin{equation}\label{lene:0010}
\begin{aligned}
    & - \int q_\varepsilon^\alpha u_\varepsilon \cdot \nabla u_\varepsilon \cdot u_{\varepsilon, t} \idx \lesssim
    \begin{cases}
	    \norm{\nabla u_\varepsilon}{\Lnorm{2}} \norm{q_\varepsilon^{\alpha/2} u_\varepsilon}{\Lnorm{4}} \norm{q_\varepsilon^{\alpha/2} u_{\varepsilon,t}}{\Lnorm{4}} & \text{when $ n = 2 $} \\
	   \norm{\nabla u_\varepsilon}{\Lnorm{2}} \norm{q_\varepsilon^{\alpha/2} u_\varepsilon}{\Lnorm{6}} \norm{q_\varepsilon^{\alpha/2} u_{\varepsilon
	   ,t}}{\Lnorm{3}} & \text{when $ n = 3$}
    \end{cases}\\
    & ~~~~ \lesssim \norm{\nabla u_\varepsilon}{\Lnorm{2}} \norm{(z+\varepsilon)^{\alpha/2} u_\varepsilon}{\Hnorm{1}} \norm{(z+\varepsilon)^{\alpha/2} u_{\varepsilon, t}}{\Lnorm{2}}^{1/2}  \norm{(z+\varepsilon)^{\alpha/2} u_{\varepsilon,t}}{\Hnorm{1}}^{1/2},
\end{aligned}
\end{equation}
where we have used Property 4 in Lemma \ref{lm:approximating-z}.
Notice, after applying the Sobolev embedding inequality and the Hardy-type inequality in Lemma \ref{lm:hardy-type-ineq}, one can derive
\begin{equation}\label{ineq:H1-embedding}
\begin{aligned}
	& \norm{(z+\varepsilon)^{\alpha/2} u_{\varepsilon,t}}{\Hnorm{1}} \lesssim \norm{(z+\varepsilon)^{\alpha/2} u_{\varepsilon,t}}{\Lnorm{2}} + \norm{(z+\varepsilon)^{\alpha/2} \nabla u_{\varepsilon, t}}{\Lnorm{2}} \\
	& ~~~~ + \norm{(z+\varepsilon)^{\alpha/2-1} u_{\varepsilon, t}}{\Lnorm{2}}
	 \lesssim \norm{(z+\varepsilon)^{\alpha/2} u_{\varepsilon,t}}{\Lnorm{2}} + \norm{(z+\varepsilon)^{\alpha/2} \nabla u_{\varepsilon,t}}{\Lnorm{2}} \\
	 & ~~~~ \lesssim \norm{q_\varepsilon^{\alpha/2} u_{\varepsilon,t}}{\Lnorm{2}} + \norm{\nabla u_{\varepsilon,t}}{\Lnorm{2}} , ~~~~ \text{and similarly} \\
	& \norm{(z+\varepsilon)^{\alpha/2} u_\varepsilon}{\Hnorm{1}} \lesssim \norm{q_\varepsilon^{\alpha/2} u_\varepsilon}{\Lnorm{2}} + \norm{\nabla u_\varepsilon}{\Lnorm{2}},
\end{aligned}
\end{equation}
provided $ \alpha/2-1 > -1/2 $, or equivalently, $ \alpha > 1 $, where we have used Property 4 in Lemma \ref{lm:approximating-z}.

On the other hand, after applying a time derivative to \subeqref{eq:approximating}{1}, the resulting equation is
\begin{equation}\label{eq:dt-ns}
    q_\varepsilon^\alpha (\dt u_{\varepsilon, t} + u_\varepsilon \cdot \nabla u_{\varepsilon, t} + u_{\varepsilon, t} \cdot \nabla u_\varepsilon) + q_\varepsilon ^\alpha \nabla p_{\varepsilon, t} = \Delta u_{\varepsilon, t}.
\end{equation}
Then after taking the $ L^2 $-inner product of \eqref{eq:dt-ns} with $ u_t $, the result is
\begin{equation}\label{lene:002}
    \dfrac{1}{2} \dfrac{d}{dt} \int q_\varepsilon^\alpha |u_{\varepsilon, t}|^2 \idx + \int |\nabla u_{\varepsilon,t}|^2 \idx = - \int q_\varepsilon^\alpha (u_{\varepsilon, t} \cdot \nabla) u_\varepsilon \cdot u_{\varepsilon,t} \idx.
\end{equation}
Similarly, the right-hand side of \eqref{eq:dt-ns} can be estimated as follows,
\begin{equation}\label{lene:0020}
\begin{aligned}
    & - \int q_\varepsilon^\alpha (u_{\varepsilon, t} \cdot \nabla) u_\varepsilon \cdot u_{\varepsilon,t} \idx = \int q_\varepsilon^\alpha (u_{\varepsilon, t} \cdot \nabla ) u_{\varepsilon, t} \cdot u_\varepsilon \idx \\
    & ~~~~ ~~~~ \lesssim
    \begin{cases}
	    \norm{q_\varepsilon^{\alpha/2} u_{\varepsilon,t}}{\Lnorm{4}} \norm{q_\varepsilon^{\alpha/2} u_\varepsilon}{\Lnorm{4}} \norm{\nabla u_{\varepsilon, t}}{\Lnorm{2}} &\text{when $ n=2 $}\\
	    \norm{q_\varepsilon^{\alpha/2} u_{\varepsilon,t}}{\Lnorm{3}} \norm{q_\varepsilon^{\alpha/2} u_\varepsilon}{\Lnorm{6}} \norm{\nabla u_{\varepsilon,t}}{\Lnorm{2}} &\text{when $ n=3$}
	 \end{cases}\\
    & ~~~~ ~~~~ \lesssim \norm{(z+\varepsilon)^{\alpha/2}u_\varepsilon}{\Hnorm{1}} \norm{(z+\varepsilon)^{\alpha/2}u_{\varepsilon, t}}{\Lnorm{2}}^{1/2} \norm{(z+\varepsilon)^{\alpha/2}u_{\varepsilon,t}}{\Hnorm{1}}^{1/2} \norm{\nabla u_{\varepsilon, t}}{\Lnorm{2}}.
\end{aligned}
\end{equation}

Therefore, combining \eqref{lene:000}, \eqref{lene:001}, \eqref{lene:0010}, \eqref{ineq:H1-embedding}, \eqref{lene:002} and \eqref{lene:0020} gives us
\begin{align*}
	& \dfrac{d}{dt} \bigl( \norm{q_\varepsilon^{\alpha/2}u_\varepsilon}{\Lnorm{2}}^2 + \norm{\nabla u_\varepsilon}{\Lnorm{2}}^2 + \norm{q_\varepsilon^{\alpha/2} u_{\varepsilon, t}}{\Lnorm{2}} \bigr) + \norm{\nabla u_\varepsilon}{\Lnorm{2}}^2 \\
	& ~~~~ + \norm{q_\varepsilon^{\alpha/2}u_{\varepsilon,t}}{\Lnorm{2}}^2
	 + \norm{\nabla u_{\varepsilon,t}}{\Lnorm{2}}^2 \lesssim \norm{q_\varepsilon^{\alpha/2}u_{\varepsilon,t}}{\Lnorm{2}}^2 \bigl( \norm{q_\varepsilon^{\alpha/2} u_\varepsilon}{\Lnorm{2}}^4 + \norm{\nabla u_\varepsilon}{\Lnorm{2}}^4 \bigr),
\end{align*}
where we have used Property 4 in Lemma \ref{lm:approximating-z} and applied Young's inequality. In particular, the above 
yields \eqref{lene:evo-01} for a short time.
\end{proof}

Next, to obtain the estimates of the spatial derivatives of $ u_\varepsilon $ requires a little work. In fact, we shall proceed with the following steps:
1. obtain estimate for the horizontal derivative; 2. obtain estimate for the pressure; 3. obtain estimate for the $ L^2 $-norm of $ \partial_{zz} u_\varepsilon $. In conclusion, we will obtain the following:
\begin{lm}\label{lm:local-ellip-est}
	In addition to the assumptions in Lemma \ref{lm:local-evol-est},
	assume that $ \alpha > 3/2 $. Then
	\begin{equation}\label{lene:H^2}
	\begin{aligned}
		& \norm{\nabla \nabla_h u_\varepsilon}{\Lnorm{2}} + \norm{q_\varepsilon^\alpha \partial_{zz}u_\varepsilon}{\Lnorm{2}} + \norm{q_\varepsilon^{2\alpha}\nabla p_\varepsilon}{\Lnorm{2}} \\
		& ~~~~ \lesssim \norm{q_\varepsilon^{\alpha/2} u_{\varepsilon,t}}{\Lnorm{2}} + \norm{u_\varepsilon}{\Hnorm{1}} + \norm{u_\varepsilon}{\Hnorm{1}}^3.
	\end{aligned}	
	\end{equation}
	In particular, \eqref{lene:H^2} together with \eqref{lene:evo-01} yields,
	\begin{equation}\label{lene:energy-total}
	\begin{aligned}
	& \sup_{0\leq t\leq T} \bigl( \norm{u_\varepsilon(t)}{\Hnorm{1}}^2 + \norm{\nabla \nabla_h u_\varepsilon(t)}{\Lnorm{2}}^2 + \norm{q_\varepsilon^\alpha\partial_{zz} u_\varepsilon(t)}{\Lnorm{2}}^2 \\
	& ~~~~ + \norm{q_\varepsilon^{\alpha/2} u_{\varepsilon,t}(t)}{\Lnorm{2}}
	 + \norm{q_\varepsilon^{2\alpha} \nabla p_\varepsilon(t)}{\Lnorm{2}} \bigr) \\
	 & ~~~~
	+ \int_0^T \bigl( \norm{\nabla u_\varepsilon(t)}{\Lnorm{2}}^2 + \norm{u_{\varepsilon,t}(t)}{\Hnorm{1}}^2\bigr) \,dt \leq C_{in},
	\end{aligned}
	\end{equation}
	where $ T $ is the same as in \eqref{lene:evo-01}, and $ C_{in} $ is some constant depending only on initial data and is independent of $ \varepsilon $.
\end{lm}
\begin{proof}
	As mentioned above, we establish the proof in three steps.
{\par\noindent\bf Step 1:} Obtain estimate for the horizontal derivative.
Taking the $ L^2 $-inner product of \subeqref{eq:approximating}{1} with $ \Delta_h u_{\varepsilon} $ implies
\begin{equation}\label{lene:003}
    \norm{\nabla \nabla_h u_{\varepsilon}}{\Lnorm{2}}^2 = \int q_\varepsilon^{\alpha} \dt u_\varepsilon \cdot \Delta_h u_{\varepsilon} \idx + \int q_\varepsilon^\alpha (u_\varepsilon \cdot \nabla) u_\varepsilon \cdot \Delta_h u_{\varepsilon} \idx.
\end{equation}
Then, applying H\"older's and the Sobolev embedding inequalities to the right-hand side of \eqref{lene:003} yields that, together with Property 4 in Lemma \ref{lm:approximating-z},
\begin{align*}
    & \int q_\varepsilon^{\alpha} \dt u_\varepsilon \cdot \Delta_h u_{\varepsilon} \idx \lesssim \norm{q_\varepsilon^{\alpha/2}\dt u_\varepsilon }{\Lnorm{2}} \norm{q_\varepsilon^{\alpha/2} \nabla \nabla_h u_{\varepsilon}}{\Lnorm{2}} \\
    & ~~~~ \lesssim \norm{q_\varepsilon^{\alpha/2}\dt u_\varepsilon }{\Lnorm{2}} \norm{\nabla \nabla_h u_{\varepsilon}}{\Lnorm{2}},\\
    & \int q_\varepsilon^\alpha (u_\varepsilon \cdot \nabla) u_\varepsilon \cdot \Delta_h u_{\varepsilon} \idx \lesssim
    \begin{cases}
    	\norm{\nabla \nabla_h u_{\varepsilon}}{\Lnorm{2}} \norm{u_\varepsilon}{\Lnorm{4}} \norm{q_\varepsilon^{\alpha} \nabla u_\varepsilon}{\Lnorm{4}} & \text{when $ n = 2 $} \\
    	\norm{\nabla \nabla_h u_{\varepsilon}}{\Lnorm{2}} \norm{u_\varepsilon}{\Lnorm{6}} \norm{q_\varepsilon^{\alpha} \nabla u_\varepsilon}{\Lnorm{3}} & \text{when $ n = 3 $}
    \end{cases}\\
    & ~~~~ \lesssim \norm{\nabla \nabla_h u_{\varepsilon}}{\Lnorm{2}} \norm{u_\varepsilon}{\Hnorm{1}}\norm{(z+\varepsilon)^{\alpha}\nabla u_\varepsilon}{\Lnorm{2}}^{1/2} \norm{(z+\varepsilon)^{\alpha}\nabla u_\varepsilon}{\Hnorm{1}}^{1/2}.
\end{align*}
Therefore \eqref{lene:003} implies
\begin{equation}\label{lene:tangential}
    \norm{\nabla \nabla_h u_{\varepsilon}}{\Lnorm{2}} \lesssim \norm{q_\varepsilon^{\alpha/2} u_{\varepsilon,t}}{\Lnorm{2}} + \norm{u_\varepsilon}{\Hnorm{1}}\norm{(z+\varepsilon)^{\alpha}\nabla u_\varepsilon}{\Lnorm{2}}^{1/2} \norm{(z+\varepsilon)^{\alpha}\nabla u_\varepsilon}{\Hnorm{1}}^{1/2}.
\end{equation}

\vspace{0.25cm}
{\par\noindent\bf Step 2:} Obtain estimate for the pressure.
Notice
\begin{equation}\label{id:laplace}
\begin{aligned}
   & q_\varepsilon^\alpha \Delta u_\varepsilon = \Delta (q_\varepsilon^\alpha u_\varepsilon) - 2 \nabla q_\varepsilon^\alpha \cdot \nabla u_\varepsilon - (\Delta q_\varepsilon^\alpha) u_\varepsilon  \\
   & ~~~~ ~~~~  = \Delta (q_\varepsilon^\alpha u_\varepsilon)
    - 2 (q_\varepsilon^\alpha)'
     \dz u_\varepsilon -
     (q_\varepsilon^\alpha)''
     u_\varepsilon .
\end{aligned}\end{equation}
Therefore, after multiplying \subeqref{eq:approximating}{1} with $ q_\varepsilon^{3\alpha} $ and applying $ \dv $ to the resulting equation, we end up with
\begin{equation}\label{ellip:p-0}
     \begin{aligned}
        & \dv(q_\varepsilon^{4\alpha} \nabla p_\varepsilon) = - \dv \bigl\lbrack q_\varepsilon^{4\alpha} ( \dt u_\varepsilon + u_\varepsilon \cdot \nabla u_\varepsilon ) - q_\varepsilon^{2\alpha} \Delta (q_\varepsilon^\alpha u_\varepsilon) \\
        & ~~~~ + 2 q_\varepsilon^{2\alpha}(q_\varepsilon^\alpha)' \dz u_\varepsilon + q_\varepsilon^{2\alpha} (q_\varepsilon^\alpha)'' u_\varepsilon  \bigr\rbrack
        = - \dv \bigl\lbrack q_\varepsilon^{4\alpha} ( \dt u_\varepsilon + u_\varepsilon \cdot \nabla u_\varepsilon ) \bigr\rbrack \\
        & ~~~~ + 2 \alpha q_\varepsilon^{2\alpha-1}q_\varepsilon' \underbrace{\bigl( \Delta (q_\varepsilon^\alpha w_\varepsilon) - q_\varepsilon^\alpha \dz \dv u_\varepsilon \bigr)}_{ = q_\varepsilon^\alpha (\Delta_h w_\varepsilon - \dz \dvh v_\varepsilon) + 2  (q_\varepsilon^{\alpha})' \dz w_\varepsilon + (q_\varepsilon^\alpha)'' w_\varepsilon} \\
        & ~~~~ - 2 (q_\varepsilon^{2\alpha} (q_\varepsilon^\alpha)')' \dz w_\varepsilon- \bigl( q_\varepsilon^{\alpha} (q_\varepsilon^\alpha)'' \bigr)' q_\varepsilon^\alpha w_\varepsilon
         .
     \end{aligned}
\end{equation}
Recall that $ p_\varepsilon $ satisfies \eqref{pressure:bc-approximating}, and the integration by parts in the following is allowed.

Thus, after taking the $ L^2 $-inner product of \eqref{ellip:p-0} with $ - p_\varepsilon $ and applying integration by parts in the resultant using the boundary conditions \eqref{bc:2d-anel-NS} and \eqref{pressure:bc-approximating}, we arrive at
\begin{equation}\label{lene:p-01}
     \norm{q_\varepsilon^{2\alpha}\nabla p_\varepsilon}{\Lnorm{2}}^2 = \underbrace{ - \int  q_\varepsilon^{4\alpha} ( \dt u_\varepsilon + u_\varepsilon\cdot \nabla u_\varepsilon) \cdot \nabla p_\varepsilon \idx }_{(I)} + (II),
\end{equation}
where
\begin{align*}
	& (II) = \int \bigl\lbrack 2\alpha q_\varepsilon^{3\alpha-1}q_\varepsilon' (\dz \dvh v_\varepsilon - \Delta_h w_\varepsilon) \\
	& ~~~~ ~~~~ + 2( (q_\varepsilon^{2\alpha} (q_\varepsilon^\alpha)')' - 2 \alpha q_\varepsilon^{2\alpha-1}q_\varepsilon' (q_\varepsilon^\alpha)' ) \dz w_\varepsilon \\
	& ~~~~ ~~~~ + ( (q_\varepsilon^\alpha (q_\varepsilon^\alpha)'')' - 2 \alpha q_\varepsilon^{\alpha-1} q_\varepsilon' (q_\varepsilon^\alpha)'' ) q_\varepsilon^\alpha w_\varepsilon \bigr\rbrack \times p_\varepsilon \idx
\end{align*}
Now we need to evaluate the right-hand side of \eqref{lene:p-01}. Indeed, applying the H\"older and the Sobolev embedding inequalities in $ (I) $ yields
\begin{align*}
    & |(I)| \lesssim \norm{q_\varepsilon^{2\alpha} \nabla p_\varepsilon}{\Lnorm{2}} \norm{q_\varepsilon^{2\alpha} u_{\varepsilon,t}}{\Lnorm{2}} \\
    & ~~~~ ~~~~ +
    \begin{cases}
    	\norm{q_\varepsilon^{2\alpha} \nabla p_\varepsilon}{\Lnorm{2}} \norm{q_\varepsilon^\alpha \nabla u_\varepsilon}{\Lnorm{4}} \norm{q_\varepsilon^{\alpha} u_\varepsilon}{\Lnorm{4}} & \text{when $ n = 2 $}\\
    	\norm{q_\varepsilon^{2\alpha} \nabla p_\varepsilon}{\Lnorm{2}} \norm{q_\varepsilon^\alpha \nabla u_\varepsilon}{\Lnorm{3}} \norm{q_\varepsilon^{\alpha} u_\varepsilon}{\Lnorm{6}} & \text{when $ n = 3 $}
    \end{cases} \\
    & ~~~~ \lesssim  \norm{q_\varepsilon^{2\alpha} \nabla p_\varepsilon}{\Lnorm{2}} \norm{q_\varepsilon^{2\alpha} u_{\varepsilon,t}}{\Lnorm{2}} \\
    & ~~~~ ~~~~ + \norm{q_\varepsilon^{2\alpha} \nabla p_\varepsilon}{\Lnorm{2}} \norm{(z+\varepsilon)^\alpha \nabla u_\varepsilon}{\Hnorm{1}}^{1/2} \norm{(z+\varepsilon)^{\alpha}\nabla u_\varepsilon}{\Lnorm{2}}^{1/2} \norm{(z+\varepsilon)^{\alpha} u_\varepsilon}{\Hnorm{1}}.
\end{align*}
To estimate $ (II) $, notice that from \subeqref{eq:approximating}{2} and \eqref{bc:2d-anel-NS}, we have
\begin{equation}\label{id:vertical-v}
	\bigl\lbrack q_\varepsilon^\alpha w_\varepsilon \bigr\rbrack (\cdot, z) = - \int_0^z \bigl\lbrack q_\varepsilon^\alpha \dvh v_\varepsilon \bigr\rbrack(\cdot,z') \,dz'.
\end{equation}
Then after substituting \eqref{id:vertical-v} in $ (II) $ and applying integration by parts, it follows,
\begin{align*}
	& (II)  = \int \nabla_h p_\varepsilon \cdot \bigl\lbrack 2\alpha q_\varepsilon^{3\alpha - 1}q_\varepsilon' (\nabla_h w_\varepsilon - \dz v_\varepsilon ) \\
    & ~~~~ + 2 \bigl( (q_\varepsilon^{2\alpha} (q_\varepsilon^\alpha)')' - 2 \alpha q_\varepsilon^{2\alpha-1}q_\varepsilon' (q_\varepsilon^\alpha)' \bigr) \bigl( q_\varepsilon^{-\alpha} \int_0^z \bigl( q_\varepsilon^\alpha  v_\varepsilon \bigr)(\cdot, z') \,dz' \bigr)_z \\
    & ~~~~ + \bigl( (q_\varepsilon^\alpha(q_\varepsilon^\alpha)'')' -2 \alpha q_\varepsilon^{\alpha-1} q_\varepsilon' (q_\varepsilon^\alpha)'' \bigr) \int_0^z \bigl( q_\varepsilon^\alpha  v_\varepsilon \bigr)(\cdot, z') \,dz' \bigr\rbrack \idx.
 \end{align*}
 Then applying Properties 4 and 5 in Lemma \ref{lm:approximating-z} and the H\"older inequality yields,
\begin{equation}\label{lene:29Jan-001}
\begin{aligned}
    & (II) \lesssim \norm{q_\varepsilon^{2\alpha} \nabla p_\varepsilon }{\Lnorm{2}} \bigl( \norm{(z+\varepsilon)^{\alpha - 1 }\nabla u_\varepsilon}{\Lnorm{2}} + \norm{(z+\varepsilon)^{\alpha - 2} v_\varepsilon}{\Lnorm{2}} \\
    & ~~~~ ~~~~ + \norm{(z+\varepsilon)^{-3} \int_0^z \bigl( q_\varepsilon^\alpha v_\varepsilon \bigr)(\cdot ,z') \,dz'}{\Lnorm{2}} \bigr) \lesssim \norm{q_\varepsilon^{2\alpha} \nabla p_\varepsilon }{\Lnorm{2}} \norm{u_\varepsilon}{\Hnorm{1}},
\end{aligned}
\end{equation}
where in the last inequality, we have applied Hardy-type inequality in the vertical direction (see Lemma \ref{lm:hardy-type-ineq}), with $ \alpha - 2 > -1/2 $, i.e., $ \alpha > 3/2 $.

Therefore, \eqref{lene:p-01} implies, for $ \alpha > 3/2 $,
\begin{equation}\label{lene:p-02}
   \begin{aligned}
   & \norm{q_\varepsilon^{2\alpha}\nabla p_\varepsilon}{\Lnorm{2}} \lesssim \norm{q_\varepsilon^{\alpha/2}u_{\varepsilon,t}}{\Lnorm{2}} + \norm{u_\varepsilon}{\Hnorm{1}}\\
  &  + \norm{(z+\varepsilon)^\alpha \nabla u_\varepsilon}{\Hnorm{1}}^{1/2} \norm{(z+\varepsilon)^{\alpha}\nabla u_\varepsilon}{\Lnorm{2}}^{1/2} \norm{u_\varepsilon}{\Hnorm{1}}.
  \end{aligned}
\end{equation}

\vspace{0.25cm}
{\par\noindent\bf Step 3:} Obtain estimate for $ \partial_{zz} u $.
We rewrite \subeqref{eq:2d-anel-NS}{1} as,
\begin{equation}\label{eq:normal-d}
    \partial_{zz} u_\varepsilon = - \Delta_h u_\varepsilon + q_\varepsilon^{\alpha} (\dt u_\varepsilon + u_\varepsilon \cdot\nabla u_\varepsilon ) + q_\varepsilon^{\alpha} \nabla p_\varepsilon.
\end{equation}
Then directly, we have
\begin{equation}\label{lene:normal-d}
\begin{aligned}
    & \norm{q_\varepsilon^{\alpha} \partial_{zz} u_\varepsilon}{\Lnorm{2}} \lesssim \norm{\Delta_h u_\varepsilon}{\Lnorm{2}} + \norm{q_{\varepsilon}^{\alpha/2} u_{\varepsilon,t}}{\Lnorm{2}} + \norm{q_\varepsilon^{2\alpha}\nabla p_\varepsilon}{\Lnorm{2}}\\
    & ~~~~ ~~~~ ~~~~ + \norm{(z+\varepsilon)^{2\alpha} u_{\varepsilon} \cdot \nabla u_\varepsilon }{\Lnorm{2}},
\end{aligned}
\end{equation}
where the last term on the right-hand side can be estimated as
\begin{align*}
    & \norm{(z+\varepsilon)^{2\alpha} u_\varepsilon \cdot \nabla u_\varepsilon}{\Lnorm{2}} \lesssim
    \begin{cases}
    	\norm{u_\varepsilon}{\Lnorm{4}}\norm{(z+\varepsilon)^{\alpha}\nabla u_\varepsilon}{\Lnorm{4}} & \text{when $ n = 2 $}\\
    	\norm{u_\varepsilon}{\Lnorm{6}}\norm{(z+\varepsilon)^{\alpha}\nabla u_\varepsilon}{\Lnorm{3}} & \text{when $ n = 3 $}
    \end{cases}\\
    & ~~~~ ~~~~ \lesssim \norm{u_\varepsilon}{\Hnorm{1}}\norm{(z+\varepsilon)^{\alpha}\nabla u_\varepsilon}{\Lnorm{2}}^{1/2} \norm{(z+\varepsilon)^{\alpha}\nabla u_\varepsilon}{\Hnorm{1}}^{1/2}.
\end{align*}

\vspace{0.25cm}

Notice,
\begin{equation}\label{lene:29Jan-002}
	\begin{aligned}
	& \norm{(z+\varepsilon)^\alpha \nabla u_\varepsilon }{\Hnorm{1}} \lesssim \norm{(z+\varepsilon)^\alpha \nabla u_\varepsilon }{\Lnorm{2}} + \norm{(z+\varepsilon)^{\alpha-1} \nabla u_\varepsilon}{\Lnorm{2}} \\
	& ~~~~ + \norm{(z+\varepsilon)^\alpha \nabla^2 u_\varepsilon }{\Lnorm{2}}.
	\end{aligned}
\end{equation}
Consequently, \eqref{lene:tangential}, \eqref{lene:p-02} and \eqref{lene:normal-d} yield \eqref{lene:H^2}. 


Now we collect \eqref{lene:evo-01} and \eqref{lene:H^2} to finish the proof. Indeed, after applying the Hardy-type inequality in Lemma \ref{lm:hardy-type-ineq}, we have the following inequalities
\begin{equation}\label{lene:29Jan-003}
\begin{aligned}
	& \norm{u_\varepsilon}{\Hnorm{1}} \lesssim \norm{(z+\varepsilon) u_\varepsilon}{\Lnorm{2}} + \norm{\nabla u_\varepsilon}{\Lnorm{2}}\lesssim \norm{(z+\varepsilon)^2 u_\varepsilon}{\Lnorm{2}} + \norm{\nabla u_\varepsilon}{\Lnorm{2}} \\
	& ~~~~  \lesssim \cdots \lesssim \norm{(z+\varepsilon)^{\alpha/2} u_\varepsilon}{\Lnorm{2}} + \norm{\nabla u_\varepsilon}{\Lnorm{2}}, \\
	& \norm{u_{\varepsilon,t}}{\Hnorm{1}} \lesssim 	\norm{(z+\varepsilon)^{\alpha/2} u_{\varepsilon,t}}{\Lnorm{2}} + \norm{\nabla u_{\varepsilon,t}}{\Lnorm{2}}.
\end{aligned}
\end{equation}
Therefore, together with Property 4 in Lemma \ref{lm:approximating-z},  \eqref{lene:evo-01}, \eqref{lene:H^2} and \eqref{lene:29Jan-003} imply the estimates in \eqref{lene:energy-total}.
\end{proof}

With \eqref{lene:energy-total},
we claim that, as $ \varepsilon \rightarrow 0^+ $, $ (u_\varepsilon, p_\varepsilon) $ converges to a strong solution to \eqref{eq:physicalvacuum-anel-NS}. Indeed, consider $ \vec{\psi} = (\psi_h, \psi_v)^\top \in C_0^\infty(\Omega;\mathbb R^{n-1} \times \mathbb R) $, where $ \psi_h $ is a scalar function, when $ n =2 $, a two-dimensional vector field, when $ n = 3 $. Here $ C_0^\infty(\Omega;\mathbb R^{n-1} \times \mathbb R) $ is the space of functions which are periodic in the horizontal variables and are of compact support in the vertical variable.
Then we have,
\begin{equation}\label{approximating-distrib}
	\begin{gathered}
		\int_0^T \int_{\Omega} \bigl\lbrack q_\varepsilon^\alpha \dt u_\varepsilon \cdot \vec{\psi} + (q_\varepsilon^\alpha u_\varepsilon \cdot \nabla) u_\varepsilon \cdot \vec{\psi} + q_\varepsilon^\alpha \nabla p_\varepsilon \cdot \vec{\psi} \bigr\rbrack \idx\,dt \\
		- \int_0^T \int_{\Omega} \Delta u_\varepsilon \cdot \vec{\psi} \idx \,dt = 0.
	\end{gathered}
\end{equation}
\eqref{lene:energy-total} implies that
 there exist $ u, p $ with
\begin{equation}\label{regularity-total}
\begin{gathered}
	 u, \nabla_h u \in L^\infty(0,T;H^1(\Omega)), ~ u \in L^2(0,T;H^1(\Omega)), \\
	  \rho_\mathrm{pv} \partial_{zz} u \in L^\infty(0,T; L^2(\Omega)), ~ \rho_\mathrm{pv}^{1/2} u_t \in L^\infty(0,T;L^2(\Omega)), \\
	  u_t \in L^2(0,T;H^1(\Omega)), ~ \rho_\mathrm{pv}^{2} \nabla p \in L^\infty(0,T;L^2(\Omega)),
\end{gathered}\end{equation}
satisfying the estimate in \eqref{total-estimate-physical-vacuum}, $ \dv(\rho_\mathrm{pv} u)=0 $, and
\begin{equation}\label{compactness-of-approximating}
\begin{aligned}
	u_\varepsilon, \nabla_h u_\varepsilon &  \buildrel\ast\over\rightharpoonup u, \nabla u & & \text{weak-$\ast$ in} ~ L^\infty(0,T;H^1(\Omega)),\\
	q_\varepsilon^\alpha \partial_{zz} u_\varepsilon, q_{\varepsilon}^{\alpha/2} u_{\varepsilon,t} &  \buildrel\ast\over\rightharpoonup \rho_\mathrm{pv}  \partial_{zz} u, \rho_\mathrm{pv}^{1/2} u_t & & \text{weak-$\ast$ in} ~ L^\infty(0,T;L^2(\Omega)),\\
	u_\varepsilon, \nabla_h u_\varepsilon , q_\varepsilon^\alpha \dz u_\varepsilon & \rightarrow u, \nabla_h u ,\rho_\mathrm{pv} \dz u & &  \text{in} ~ C([0,T];L^2(\Omega)),\\
	u_{\varepsilon,t}, u_\varepsilon & \rightharpoonup u_t, u & & \text{weakly in} ~ L^2(0,T;H^1(\Omega)), \\
	q_{\varepsilon}^{2\alpha} \nabla p_\varepsilon & \buildrel\ast\over\rightharpoonup \rho_\mathrm{pv}^{2} \nabla p & & \text{weak-$\ast$ in} ~ L^\infty(0,T;L^2(\Omega)),
\end{aligned}
\end{equation}
where we have used Property 1 in Lemma \ref{lm:approximating-z}.
Thus we have $ \lim\limits_{t\rightarrow 0^+} u = u_{in}
$, and after passing the limit with $ \varepsilon \rightarrow 0^+ $, in \eqref{approximating-distrib}, we have
\begin{equation}\label{approximating-distrib-limit}
	\begin{gathered}
		\int_0^T \int_{\Omega} \bigl\lbrack \rho_\mathrm{pv} \dt u \cdot \vec{\psi} + (\rho_\mathrm{pv} u \cdot \nabla) u \cdot \vec{\psi} + \rho_\mathrm{pv} \nabla p \cdot \vec{\psi} \bigr\rbrack \idx\,dt \\
		 - \int_0^T \int_{\Omega} \Delta u  \cdot \vec{\psi} \idx \,dt = 0,
	\end{gathered}
\end{equation}
which verifies that $ (u,p)|_{t\in (0,T)} $ is a solution to \eqref{eq:physicalvacuum-anel-NS} in $ \Omega $. We recall that $ \vec{\psi} $ is chosen such that its support is away from $ \lbrace z = 0 , 1 \rbrace $.
Moreover, it is easy to verify
\begin{equation}\label{pressure-viscosity}
	- \Delta u + \rho_\mathrm{pv} \nabla p = - \rho_\mathrm{pv} \dt u - \rho_\mathrm{pv} u\cdot \nabla u \in L^\infty(0,T;L^2(\Omega)).
\end{equation}

On the other hand, the trace theorem implies that $ \rho_\mathrm{pv}\dz v\big|_{z=0}, \dz v\big|_{z=1}, w\big|_{z=0,1} \in L^2(0,T;L^2(2\mathbb T^{n-1})) $, thanks to the regularity in \eqref{regularity-total}. Thus
$$ \rho_\mathrm{pv} \dz v\big|_{z=0} = 0,~ \dz v\big|_{z=1} = 0,~ w\big|_{z=0,1} = 0 ~~~~ \text{in} ~ L^2(0,T;L^2(2\mathbb T^{n-1})).$$
To verify the boundary condition $ \dz v\big|_{z=0} = 0 $ in \eqref{bc:2d-anel-NS}, consider $ \psi_{h,\varepsilon} (x,z,t) := \bigl(1- c_\varepsilon q_\varepsilon (z) \bigr) \psi_1(x,t) $ with $ \psi_1 \in C^\infty (2\mathbb T^{n-1} \times [0,T]; \mathbb R^{n-1}) $ for some constant $ c_\varepsilon $ satisfying
$$
\int_0^1 \bigl(1 - c_\varepsilon q_\varepsilon (z) \bigr) q_\varepsilon^\alpha(z) \,dz = 0, ~~ i.e., ~ c_\varepsilon := \dfrac{ \int_0^1 q_\varepsilon^\alpha(\xi) \,d\xi }{ \int_0^1 q_\varepsilon^{\alpha+1}(\xi) \,\xi }.
$$
Consider $ \vec{\psi}_{\varepsilon} := (\psi_{h,\varepsilon}, \psi_{v,\varepsilon} ) $ with $ \psi_{h,\varepsilon} $ given as above and
\begin{equation*}
\psi_{v,\varepsilon}(x,z,t) := - q_\varepsilon^{-\alpha}(z) \int_0^z q_\varepsilon^{\alpha}(\xi) \dvh \psi_{h,\varepsilon}(x,\xi,t) \,d\xi.
 \end{equation*}
 Then $ \vec{\psi}_\varepsilon $ satisfies $ \dv \bigl( q_\varepsilon^\alpha \vec{\psi}_\varepsilon \bigr) = 0 $, $ \psi_{v,\varepsilon}\big|_{z=0,1} = 0$, and as $ \varepsilon \rightarrow 0^+ $, $ \vec{\psi}_{\varepsilon} \rightarrow \vec{\psi}_{0} = (\psi_{h,0}, \psi_{v,0}) $ uniformly, where
 \begin{gather*}
 	\psi_{h,0}(x,z,t) = \bigl(1 - \dfrac{\int_0^1 \xi^\alpha(2-\xi)^\alpha \,d\xi }{ \int_0^1 \xi^{\alpha+1} (2-\xi)^{\alpha+1} \,d\xi} z(2-z) \bigr) \psi_1(x,t) ~~~~ \text{and}\\
 	\psi_{v,0}(x,z,t) =  \biggl( \dfrac{\int_0^1 \xi^\alpha(2-\xi)^\alpha \,d\xi }{ \int_0^1 \xi^{\alpha+1} (2-\xi)^{\alpha+1} \,d\xi} \dfrac{\int_0^z \xi^{\alpha+1}(2-\xi)^{\alpha+1} \,d\xi}{z^\alpha (2-z)^\alpha} \\
 	~~~~ ~~~~- \dfrac{\int_0^z \xi^{\alpha}(2-\xi)^{\alpha} \,d\xi}{z^\alpha (2-z)^\alpha} \biggr) \dvh \psi_1(x,t)  .
 \end{gather*}
Now we choose $ \vec{\psi} = \vec{\psi}_{\varepsilon} $ in \eqref{approximating-distrib}. After applying integration by parts, we arrive at
 \begin{gather*}
 	\int_0^T \int_{\Omega} \bigl\lbrack q_\varepsilon^\alpha \dt u_\varepsilon \cdot \vec{\psi}_\varepsilon + (q_\varepsilon^\alpha u_\varepsilon \cdot \nabla) u_\varepsilon \cdot \vec{\psi}_\varepsilon \bigr\rbrack \idx\,dt \\
		 = - \int_0^T \int_{\Omega} \nabla u_\varepsilon : \nabla \vec{\psi}_\varepsilon \idx \,dt + \int_0^T  \int_{2\mathbb T^{n-1}} (\dz v_\varepsilon \cdot \psi_{h,\varepsilon})|_{z=1} \,dx\,dt\\
		 - \int_0^T  \int_{2\mathbb T^{n-1}} (\dz v_\varepsilon \cdot \psi_{h,\varepsilon})|_{z=0} \,dx\,dt,
 \end{gather*}
 which, together with \eqref{lene:energy-total} and the trace theorem, implies that
 \begin{equation}\label{approximating-bc-dzv}
 	\begin{aligned}
 		&(1 - c_\varepsilon q_\varepsilon(0))   \int_0^T  \int_{2\mathbb T^{n-1}} (\dz v_\varepsilon \cdot \psi_1(x,t))|_{z=0} \,dx\,dt \\
 		& ~~  = \int_0^T  \int_{2\mathbb T^{n-1}} (\dz v_\varepsilon \cdot \psi_{h,\varepsilon})|_{z=0} \,dx\,dt \leq  C_{in} \norm{\vec{\psi}_{\varepsilon}}{L^2(0,T;H^1(\Omega))} \\
 		& ~~~~ \leq C_{in} \norm{\psi_1}{L^2(0,T;H^2(2\mathbb T^{n-1}))}.
 	\end{aligned}
 \end{equation}
 Notice that, Property 4 in Lemma \ref{lm:approximating-z} implies $ 1- c_\varepsilon q_\varepsilon(0) > 1/2 $ for $ \varepsilon $ small enough.
Thus, \eqref{approximating-bc-dzv} yields that
$ \lbrace \dz v_{\varepsilon}\big|_{z=0} \rbrace $ is uniformly bounded in $ L^2(0,T;(H^2(2\mathbb T^{n-1}))^*) $ and thus as $ \varepsilon \rightarrow 0^+ $,
$$0 = \dz v_{\varepsilon}\big|_{z=0} \rightharpoonup  \dz v\big|_{z=0}  ~~ \text{weakly in}~ L^2(0,T;(H^2(2\mathbb T^{n-1}))^*). $$ In particular,  $ \dz v\big|_{z=0} = 0 ~ \text{in}~ L^2(0,T;L^2(2\mathbb T^{n-1})) $ and so we have verified the boundary conditions in \eqref{bc:2d-anel-NS}.

In addition, consider $ \psi_h \in L^2(0,T;H^1(\Omega)) $ and $ \psi_v\in  L^2(0,T;H^1_0(\Omega)) $.
Then, $ ( \Delta v,\Delta w)^\top \in (L^2(0,T;H^1(\Omega)))^* \times (L^2(0,T;H^1_0(\Omega)))^* $
is a functional which acts on $ (\psi_h, \psi_v)^\top $ by the duality
$$ \langle (\Delta v, \Delta w)^\top, (\psi_h,\psi_v)^\top \rangle = - \int_0^T \int_\Omega \nabla v : \nabla \psi_h \idx \,dt - \int_0^T \int_\Omega \nabla w \cdot \nabla \psi_v \idx \,dt.
$$
Moreover, from \eqref{pressure-viscosity}, one can infer that $ \rho_\mathrm{pv} \nabla p $ is a functional acting on $ (\psi_h, \psi_v)^\top $. In particular, if $ \dv(\rho_\mathrm{pv} \vec{\psi}) = 0 $, we have
  $$ \langle \rho_\mathrm{pv} \nabla p,(\psi_h, \psi_v)^\top \rangle = -\int_0^T \int_\Omega p \dv (\rho_\mathrm{pv} \vec{\psi})\idx\,dt = 0. $$
Consequently, the regularity of $ u $, as in \eqref{regularity-total}, allows us to consider the action of \subeqref{eq:physicalvacuum-anel-NS}{1} on $ u $. That is, the following equation holds in $ \mathcal D'(0,T) $,
\begin{equation*}
	\dfrac{1}{2} \dfrac{d}{dt} \int \rho_\mathrm{pv} u^2 \idx + \int |\nabla u|^2 \idx = 0.
\end{equation*}
Thus we have the energy identity, for any $ t \in [0,T] $,
\begin{equation}\label{id:energy-identity}
	\norm{\rho_\mathrm{pv}^{1/2} u(t)}{\Lnorm{2}}^2 + 2 \int_0^t \norm{\nabla u(s)}{\Lnorm{2}}^2\,ds = \norm{\rho_\mathrm{pv}^{1/2} u_{in}}{\Lnorm{2}}^2.
\end{equation}

With such properties, we are able to show the uniqueness of solutions. Consider $ u_1, u_2 $ being solutions to \eqref{eq:physicalvacuum-anel-NS} as above with initial data $ u_{in,1}, u_{in,2} $. Also, denote $ T \in (0,\infty) $ as the existence time for both solutions. Then consider the actions of  \subeqref{eq:physicalvacuum-anel-NS}{1} for $ u_1 $ with $ u_2 $ and \subeqref{eq:physicalvacuum-anel-NS}{1} for $ u_2 $ with $ u_1 $.  Summing up the results leads to, for any $ t\in [0,T] $,
\begin{equation}\label{uniqueness-interaction}
\begin{aligned}
	& \int_\Omega \rho_\mathrm{pv} u_1(t) \cdot  u_2(t) \idx + \int_0^t \int_\Omega 2 \nabla u_1(s) : \nabla u_2(s) \idx \,ds   = \int_\Omega \rho_\mathrm{pv} u_{in,1} \cdot u_{in,2} \idx \\
	& ~~~ - \int_0^t \int_\Omega \rho_\mathrm{pv} ((u_1(s) \cdot \nabla) u_1(s) \cdot u_2 (s)+ (u_2(s) \cdot \nabla) u_2(s) \cdot u_1(s) ) \idx\,ds.
\end{aligned}
\end{equation}
Notice, after applying \subeqref{eq:physicalvacuum-anel-NS}{2} and integration by parts, we have
\begin{align*}
	& \int_\Omega \rho_\mathrm{pv} ((u_1 \cdot \nabla) u_1 \cdot u_2 + (u_2 \cdot \nabla) u_2 \cdot u_1 )\idx \\
	& ~~~~ = \int_\Omega \rho_\mathrm{pv} ( (u_1 - u_2) \cdot \nabla) (u_1-u_2) \cdot u_2 \idx\\
	& ~~~~ \leq
	 \begin{cases}
	 \norm{z^\alpha (u_1-u_2)}{\Lnorm{4}} \norm{\nabla (u_1 - u_2)}{\Lnorm{2}} \norm{u_2}{\Lnorm{4}} & \text{when} ~ n = 2,\\
	  \norm{z^\alpha(u_1-u_2)}{\Lnorm{3}} \norm{\nabla (u_1 - u_2)}{\Lnorm{2}} \norm{u_2}{\Lnorm{6}} & \text{when} ~ n = 3,\\
	 \end{cases}\\
	& ~~~~ \leq C \norm{z^\alpha(u_1-u_2)}{\Lnorm{2}}^{1/2} \norm{z^\alpha (u_1 - u_2)}{\Hnorm{1}}^{1/2} \norm{\nabla (u_1-u_2)}{\Lnorm{2}} \norm{u_2}{\Hnorm{1}}\\
	& ~~~~ \leq C \norm{\rho_\mathrm{pv}(u_1-u_2)}{\Lnorm{2}}^{1/2} ( \norm{\rho_\mathrm{pv} (u_1 - u_2)}{\Lnorm{2}}+ \norm{\rho_\mathrm{pv} \nabla (u_1 - u_2)}{\Lnorm{2}})^{1/2} \\
	& ~~~~ ~~~~ \times \norm{\nabla (u_1-u_2)}{\Lnorm{2}} \norm{u_2}{\Hnorm{1}},
\end{align*}
where the last inequality follows by applying Hardy's inequality in Lemma \ref{lm:hardy-type-ineq} and the fact that $ \rho_\mathrm{pv} \simeq z^\alpha $ for $ z \in (0,1/2) $.

Therefore, \eqref{uniqueness-interaction}, together with the energy identity \eqref{id:energy-identity} for $ u_1, u_2 $, implies
\begin{align*}
	& \norm{\rho_\mathrm{pv}^{1/2} (u_1(t) - u_2(t)) }{\Lnorm{2}}^2 + 2 \int_0^t \norm{\nabla (u_1(s) - u_2(s))}{\Lnorm{2}}^2 \,ds \\
	& ~~~~ \leq \norm{\rho_\mathrm{pv}^{1/2} (u_{in,1} - u_{in,2}) }{\Lnorm{2}}^2 + \int_0^t \norm{\nabla (u_1(s)-u_2(s))}{\Lnorm{2}}^2 \,ds \\
	& ~~~~ + C \int_0^t (1 + \norm{u_2(s)}{\Hnorm{1}}^4) \norm{\rho_\mathrm{pv}^{1/2} (u_1(s)-u_2(s))}{\Lnorm{2}}^2 \,ds.
\end{align*}
Then applying Gr\"onwall's inequality yields,
\begin{equation}\label{stability}
	\begin{aligned}
	& \sup_{0\leq t\leq T}\norm{\rho_\mathrm{pv}^{1/2} (u_1(t) - u_2(t)) }{\Lnorm{2}}^2 + \int_0^T \norm{\nabla (u_1(s) - u_2(s))}{\Lnorm{2}}^2 \,ds \\
	& ~~~~ \leq C_{in,T} \norm{\rho_\mathrm{pv}^{1/2} (u_{in,1} - u_{in,2}) }{\Lnorm{2}}^2,
	\end{aligned}
\end{equation}
for some constant $ C_{in,T} $ depending on $ T $ and the initial data $ u_{in,1},u_{in,2} $. In particular, this implies the uniqueness of solutions.

We remark that the above uniqueness argument is similar to the one used by J. Serrin for weak-strong uniqueness of three-dimensional Navier--Stokes equations in \cite{Serrin1963}.


\section{Global-in-time a priori estimates when $ n = 2 $}\label{sec:global-2d}

In this and the following sections, we present some global-in-time a priori estimates of solutions to \eqref{eq:physicalvacuum-anel-NS}. These arguments can be rigorously justified following the arguments in sections  \ref{sec:local-solutions-nonsingular} and \ref{sec:physical-vacuum}.
The estimates of solutions to \eqref{eq:nonsingular-anel-NS} are similar, and will be omitted.

Notice that, the regularity of $ u $ in \eqref{regularity-total} allows us to take the following actions.
Taking the $ L^2 $-inner product of \subeqref{eq:physicalvacuum-anel-NS}{1} with $ u $ implies, as in \eqref{lene:000}, 
\begin{equation}\label{lene:0000}
    \dfrac{1}{2} \dfrac{d}{dt} \int \rho_\mathrm{pv} |u|^2 \idx + \int |\nabla u|^2 \idx = 0.
\end{equation}
As in \eqref{lene:001}, we also have,
\begin{equation}\label{lene:0001}
    \dfrac{1}{2} \dfrac{d}{dt} \int |\nabla u|^2 \idx + \int \rho_\mathrm{pv} |u_t|^2 \idx = - \int \rho_\mathrm{pv} u \cdot \nabla u \cdot u_t \idx.
\end{equation}
The right-hand side of \eqref{lene:0001} can be estimated as follows, due to the fact that $ \rho_\mathrm{pv} \lesssim z^\alpha $,
\begin{align*}
    & - \int \rho_\mathrm{pv} u \cdot \nabla u \cdot u_t \idx \leq \dfrac{1}{2} \int \rho_\mathrm{pv} |u_t|^2 \idx + C \norm{z^\alpha u}{\Lnorm{\infty}}^2 \int |\nabla u|^2\idx\\
    & ~~~~ \leq \dfrac{1}{2} \int \rho_\mathrm{pv} |u_t|^2 \idx + C \int |\nabla u|^2\idx \cdot ( \norm{z^\alpha u}{H^1}^2 + 1 ) \log ( e + \norm{z^\alpha u}{H^2}^2),
\end{align*}
where we have applied Young's inequality and the two-dimensional Brezis-Gallouate-Wainger inequality (see, e.g., \cite{brezis-1,brezis-2}).

Meanwhile, the same arguments as \eqref{eq:dt-ns} through \eqref{lene:002} imply the similar estimate to \eqref{lene:002}, i.e.,
\begin{equation}{\label{lene:0002}}
    \dfrac{1}{2} \dfrac{d}{dt} \int \rho_\mathrm{pv} |u_t|^2 \idx + \int |\nabla u_t|^2 \idx = - \int \rho_\mathrm{pv} (u_t \cdot \nabla) u \cdot u_t \idx,
\end{equation}
where
\begin{align*}
	& - \int \rho_\mathrm{pv} (u_t \cdot \nabla) u \cdot u_t \idx = \int \rho_\mathrm{pv} (u_t \cdot \nabla) u_t \cdot u \idx \leq \dfrac{1}{2} \int  |\nabla u_t|^2 \idx \\
	& ~~~~ + C \norm{z^\alpha u}{\Lnorm{\infty}}^2 \int \rho_\mathrm{pv} |u_t|^2 \idx \leq \dfrac{1}{2} \int  |\nabla u_t|^2 \idx \\
	& ~~~~ + C  \int \rho_\mathrm{pv} |u_t|^2 \idx\cdot ( \norm{z^\alpha u}{H^1}^2 + 1 ) \log ( e + \norm{z^\alpha u}{H^2}^2).
\end{align*}
In addition, due to the fact that $ z^\alpha \lesssim \rho_\mathrm{pv} $ for $ z \in (0,1/2) $, applying Hardy's inequality in Lemma \ref{lm:hardy-type-ineq} yields, with $ \alpha > 3/2 $,
\begin{equation}\label{applying-hardy}
	\begin{aligned}
	& \norm{z^\alpha u}{\Hnorm{1}} \lesssim \norm{z^{\alpha/2} u}{\Lnorm{2}} + \norm{\nabla u}{\Lnorm{2}} \lesssim \norm{\rho_\mathrm{pv}^{1/2} u}{\Lnorm{2}} + \norm{\nabla u}{\Lnorm{2}}, \\
	& \norm{z^\alpha u}{\Hnorm{2}} \lesssim \norm{z^\alpha \nabla^2 u}{\Lnorm{2}} + \norm{z^{\alpha-1} \nabla u}{\Lnorm{2}} + \norm{z^{\alpha-2} u}{\Lnorm{2}}\\
	& ~~~~ ~~~~ \lesssim \norm{\nabla \nabla_h u}{\Lnorm{2}} + \norm{\nabla u}{\Lnorm{2}} + \norm{z^{\alpha/2} u}{\Lnorm{2}} + \norm{z^\alpha \partial_{zz} u}{\Lnorm{2}}\\
	& ~~~~ ~~~~ \lesssim \norm{\nabla \nabla_h u}{\Lnorm{2}} + \norm{\nabla u}{\Lnorm{2}} + \norm{\rho_\mathrm{pv}^{1/2} u}{\Lnorm{2}} + \norm{\rho_\mathrm{pv} \partial_{zz} u}{\Lnorm{2}}.
	\end{aligned}
\end{equation}

On the other hand, similar to \eqref{lene:H^2},
\begin{equation}\label{lene:0H^2}
	\begin{aligned}
		& \norm{\nabla\nabla_h u}{\Lnorm{2}} + \norm{\rho_\mathrm{pv} \partial_{zz} u}{\Lnorm{2}} + \norm{\rho_\mathrm{pv}^2 \nabla p}{\Lnorm{2}} \\
		& ~~~~ \lesssim \norm{\rho_\mathrm{pv}^{1/2} u_t}{\Lnorm{2}} + \norm{u}{\Hnorm{1}} + \norm{u}{\Hnorm{1}}^3.
	\end{aligned}
\end{equation}
Then together with \eqref{lene:0001}, \eqref{lene:0002} ,
we arrive at
\begin{equation}\label{ene:total_ene_ineq}
	\dfrac{d}{dt} E(t) \lesssim E(t) ( 1 + \norm{\rho_\mathrm{pv}^{1/2} u(t)}{\Lnorm{2}}^2 +\norm{\nabla u(t)}{\Lnorm{2}}^2) \log (\norm{\rho_\mathrm{pv}^{1/2} u(t)}{\Lnorm{2}}^{6} + E^3(t) ),
\end{equation}
where
\begin{equation*}
	E(t):= e + \int |\nabla v(t)|^2 \idx + \int \rho_\mathrm{pv} |u_t(t)|^2 \idx.
\end{equation*}
Also, \eqref{lene:0000} implies, for any $ T \in (0,\infty) $,
\begin{equation}\label{19-april-01}
	\sup_{0\leq t\leq T} \norm{\rho_\mathrm{pv}^{1/2} u(t)}{\Lnorm{2}}^2 + \int_0^T \norm{\nabla u(t)}{\Lnorm{2}}^2 \,dt \leq C_{in},
\end{equation}
for some positive constant $ C_{in} $ independent of $ T $. Therefore, \eqref{ene:total_ene_ineq} implies that
\begin{equation*}
	\dfrac{d}{dt} \log E(t) \lesssim (1 +  \norm{\rho_\mathrm{pv}^{1/2} u(t)}{\Lnorm{2}}^2 +\norm{\nabla u(t)}{\Lnorm{2}}^2) \log E(t).
\end{equation*}
Thus applying Gr\"onwall's inequality yields, together with \eqref{19-april-01},
\begin{equation}\label{19-april-02}
\begin{aligned}
	& \sup_{0\leq t \leq T} \log \log E(t) \leq  C \int_0^T (1 + \norm{\rho_\mathrm{pv}^{1/2} u(s)}{\Lnorm{2}}^2 +\norm{\nabla u(s)}{\Lnorm{2}}^2) \,ds\\
	& ~~~~ + \log \log E(0) \leq  C_{in}(T+1) + \log \log E(0).
\end{aligned}
\end{equation}
for some constant $ C_{in} $ depending only on the initial data. \eqref{19-april-01} and \eqref{19-april-02} imply the global well-posedness.

\section{Small data global-in-time a priori estimates when $ n = 3 $}\label{sec:global-3d}

Similarly, the estimates in \eqref{lene:0000}, \eqref{lene:0001} and \eqref{lene:0002} hold. That is,
\begin{gather*}
    \dfrac{1}{2} \dfrac{d}{dt} \int \rho_\mathrm{pv} |u|^2 \idx + \int |\nabla u|^2 \idx = 0, \tag{\ref{lene:0000}} \\
	\tag{\ref{lene:0001}}
    \dfrac{1}{2} \dfrac{d}{dt} \int |\nabla u|^2 \idx + \int \rho_\mathrm{pv} |u_t|^2 \idx = - \int \rho_\mathrm{pv} u \cdot \nabla u \cdot u_t \idx, \\
    {\tag{\ref{lene:0002}}}
    \dfrac{1}{2} \dfrac{d}{dt} \int \rho_\mathrm{pv} |u_t|^2 \idx + \int |\nabla u_t|^2 \idx = - \int \rho_\mathrm{pv} (u_t \cdot \nabla) u \cdot u_t \idx.
\end{gather*}
We estimate the nonlinearities on the right-hand side of \eqref{lene:0001} and \eqref{lene:0002} as follows,
\begin{align*}
	& - \int \rho_\mathrm{pv} u \cdot \nabla u \cdot u_t \idx \leq \dfrac{1}{2} \int \rho_\mathrm{pv} |u_t|^2 \idx + C \norm{z^\alpha u}{\Hnorm{2}}^2 \int |\nabla u|^2\idx, \\
	& - \int \rho_\mathrm{pv} (u_t \cdot \nabla) u \cdot u_t \idx = \int \rho_\mathrm{pv} (u_t \cdot \nabla) u_t \cdot u \idx \leq \dfrac{1}{2} \int  |\nabla u_t|^2 \idx \\
	& ~~~~ + C \norm{z^\alpha u}{\Hnorm{2}}^2 \int \rho_\mathrm{pv} |u_t|^2 \idx.
\end{align*}
Then, after denoting
\begin{equation*}
	E(t): = \int \rho_\mathrm{pv} |u|^2 \idx + \int |\nabla u|^2 \idx + \int \rho_\mathrm{pv} |u_t|^2 \idx,
\end{equation*}
\eqref{lene:0000}, \eqref{lene:0001}, \eqref{lene:0002}, \eqref{applying-hardy} and \eqref{lene:0H^2} imply
\begin{equation*}
	\dfrac{d}{dt} E(t) + (1 - E^6 ) \int (|\nabla u|^2 + \rho_\mathrm{pv} |u_t|^2 + |\nabla u_t|^2) \idx \leq 0,
\end{equation*}
which implies, for $ E(0) $ small enough,
\begin{equation*}
	\sup_{0\leq t < \infty}E(t) \leq E(0).
\end{equation*}
Thus we have shown the global well-posedness with small initial data.

\section*{Acknowledgements}
This work was supported in part by the Einstein Stiftung/Foundation - Berlin, through the Einstein Visiting Fellow Program, and by the John Simon Guggenheim Memorial Foundation.

\bibliographystyle{plain}

\end{document}